\documentclass{amsart}

\usepackage[a4paper,text={140mm,240mm},centering,headsep=5mm,footskip=10mm]{geometry}

\usepackage{amssymb,amsmath} 
\usepackage{color}
\usepackage{graphicx}
\usepackage[matrix,arrow, curve]{xy}
\usepackage{algorithm}
\usepackage[noend]{algorithmic}

\usepackage[hypertexnames=false,linkcolor=black, colorlinks=true, citecolor=black, filecolor=black,urlcolor=black]{hyperref}
\usepackage{float}
\newfloat{algorithm}{t}{lop}

\usepackage{tikz}
\usetikzlibrary{automata,arrows,shapes,petri}
\newlength{\lw}
\newlength{\st}
\newlength{\qs}
\newlength{\ndL}
\tikzstyle{every picture}+=[auto]
\tikzstyle{every picture}+=[bend angle=10]
\tikzstyle{every picture}+=[join=round]
\tikzstyle{every picture}+=[cap=butt]
\tikzstyle{every picture}+=[line width=\lw]
\tikzstyle{every picture}+=[double distance=2\lw]
\tikzstyle{every picture}+=[shorten >=\st]
\tikzstyle{every picture}+=[node distance=\ndL]
\tikzstyle{every loop}=[->,shorten >=\st]
\tikzstyle{placeL}=[circle,draw,minimum size=2\qs]
\tikzstyle{transitionL}=[rectangle,draw,minimum width=3\qs,minimum height=2\qs]
\tikzstyle{invisible}=[draw=none,inner sep=0pt,minimum height=0pt]
\tikzstyle{startplaceL}=[semicircle,draw,minimum size=\qs]
\tikzstyle{endplaceL}=[semicircle,draw,minimum size=\qs,shape border rotate=180]
\newenvironment{petrinet}[1][]{\begin{center}\begin{tikzpicture}}{\end{tikzpicture}\end{center}}

\newcommand{\NN}{{\mathbb{N}}}
\newcommand{\QQ}{{\mathbb{Q}}}

\newcommand{\ZZ}{{\mathbb{Z}}}

\newcommand{\ba}{ {\boldsymbol a} }
\newcommand{\bb}{ {\boldsymbol b} }
\newcommand{\bx}{ {\underline{x}} } 
\newcommand{\bX}{ {\underline{X}} }

\newcommand{\bT}{ {\underline{T}} }
\newcommand{\bs}{ {\underline{s}} }

\newcommand{\cB}{{\mathcal{B}}}
\newcommand{\cD}{{\mathcal{D}}}
\newcommand{\cO}{{\mathcal{O}}}
\newcommand{\fP}{{\mathfrak P}}
\newcommand{\fM}{{\mathfrak M}}
\newcommand{\fm}{{\mathfrak m}}
\newcommand{\fn}{{\mathfrak n}}

\newcommand{\embdim}{\mathrm{emb.dim}} 
\newcommand{\Ord}[1]{\mathcal{V}_{#1}} 
\newcommand{\OOrd}[1]{\mathcal{V}^O_{#1}} 
\newcommand{\maxOrd}{{\mathrm{max}\mbox{-}\ord}} 
\newcommand{\maxNu}{{\mathrm{max}\mbox{-}\nu}} 
\newcommand{\OmaxNu}{{\mathrm{max}\mbox{-}\nu^O}} 
\newcommand{\MaxOrd}{{\mathrm{Max}\mbox{-}\ord }} 
\newcommand{\MaxNu}{{\mathrm{Max}\mbox{-}\nu }} 
\newcommand{\OMaxNu}{{\mathrm{Max}\mbox{-}\nu^O}} 
\newcommand{\Max}{\mathrm{Max}} 
\newcommand{\Ini}{{\mathrm{In}}} 
\newcommand{\ini}{{\mathrm{in}}} 
\newcommand{\gr}{{\mathrm{gr}}} 
\newcommand{\ord}{{\mathrm{ord}}} 
\newcommand{\nut}{{\nu_{\sf ref}}} 
\newcommand{\Onut}{{\nu_{\sf ref}^O}} 
\newcommand{\Bl}{{B\ell}} 
\newcommand{\Spec}{{\mathrm{Spec}}}

\newcommand{\Sing}{{\mathrm{Sing}}}
\newcommand{\good}{{weakly permissible }}

\theoremstyle{plain}
\newtheorem{Thm}{Theorem}[section]
\newtheorem{Lem}[Thm]{Lemma}
\newtheorem{Prop}[Thm]{Proposition}

\theoremstyle{definition}
\newtheorem{Def}[Thm]{Definition}
\newtheorem{Rk}[Thm]{Remark}
\newtheorem{Ex}[Thm]{Example}

\newtheorem{Not_num}[Thm]{Notation}
\newtheorem{Obs}[Thm]{Observation}
\newtheorem{Construction}[Thm]{Construction}

\newcommand{\xch}[1]{{\color{black} #1}}
\newcommand{\farbe}{black}

\numberwithin{equation}{section}

\title[Embedded Desingularization for arithmetic surfaces]
{
	Embedded Desingularization for arithmetic surfaces -- toward a parallel implementation
}

\author{Anne Fr\"uhbis-Kr\"uger}
\address{A.FK., Institut f\"ur Algebraische Geometrie,
	Leibniz Universit\"at Hannover,
	Welfengarten 1,
	30167 Hannover,
	Germany}
\curraddr{Institut f\"ur Mathematik, Carl von Ossietzky Universit\"at Oldenburg, 26111 Oldenburg, Germany}
\email{anne.fruehbis-krueger@uni-oldenburg.de}

\author{Lukas Ristau}
\address{L.R., Competence Center High Performance Computing,
	Fraunhofer Institute for Industrial Mathematics ITWM,
	Fraunhofer-Platz 1,
	67663 Kaiserslautern,
	Germany}
\curraddr{Department of Mathematics,
	University of Kaiserslautern,
	Erwin-Schr\"odinger-Str.,
	67663 Kaiserslautern,
	Germany}
\email{ristau@mathematik.uni-kl.de}
\thanks{\xch{L.R.~was partially supported by DFG-grant SFB-TRR 195 ``Symbolic Tools in Mathematics and their Application'', project II.5 ``{\sc Singular}: A new level of abstraction and performance''}}

\author{Bernd Schober}
\address{B.S., Institut f\"ur Algebraische Geometrie,
	Leibniz Universit\"at Hannover,
	Welfengarten 1,
	30167 Hannover,
	Germany}
\curraddr{Institut f\"ur Mathematik, Carl von Ossietzky Universit\"at Oldenburg, 26111 Oldenburg, Germany}
\email{bernd.schober@uni-oldenburg.de}


\begin{document}

\begin{abstract}	
We present an algorithmic embedded desingularization of 
arithmetic surfaces bearing in mind implementability.
Our algorithm is based on work by Cossart-Jannsen-Saito, 
though our variant uses a refinement of the order 
instead of the Hilbert-Samuel function as a measure for the complexity of 
the singularity.
We particularly focus on aspects arising when working in mixed characteristics. 
Furthermore, we exploit the algorithm's natural 
parallel structure rephrasing it in terms of Petri nets for use in the 
parallelization environment GPI-Space with {\sc Singular} as computational 
back-end.
\end{abstract}

\maketitle

\label{lukas}

\section{Introduction}

Resolution of singularities in all its facets has captured the attention of 
many algebraic geometers since the end of the 19th century. Hironaka's 
monumental work \cite{Hiro64} solved the long standing problem of 
desingularization in characteristic zero in arbitrary dimension in 1964, 
whereas the problem is still wide open in positive and mixed characteristic
in higher dimension. With the aim to device a viable approach in positive
characteristic, Hironaka's characteristic zero approach has been analyzed and 
rephrased over the decades.
Starting with the first constructive formulations,
(see in particular \cite{BM89}, \cite{OrlandoConstr}) 
and even more 
algorithmic simplified approaches (as for example \cite{BEV,EHa,Jarek})
to
the point that nowadays there are accessible, concise introductions to the subject
available (like \cite{Cut}, \cite{Kol}). 
Around the turn of the century the
algorithmization of resolution of singularities had reached a level which
allowed first prototype implementations \cite{BSch}, \cite{FKP}, 
of which
the latter was sufficiently efficient to also permit various applications 
\cite{ResZETA}, \cite{FKPnato}, \cite{FKzeta}.

In positive characteristic, there exist only results in small dimensions and in very special cases \cite{BMbinomial}, \cite{Blanco1}, \cite{Blanco2}, \cite{KollarToroidal}, \cite{BerndPartial}.
While the surface case has been
treated in many different ways \cite{A}, \cite{AngOrlDim2}, \cite{CGO}, \cite{CJS}, \cite{InvDim2}, \cite{CutSurf}, \cite{HiroBowdoin}, \cite{IFPdim2}, \cite{Lip}, it was not until the
end of the last decade that the dimension 3 case (non-embedded) was solved in full 
generality \cite{CP1}, \cite{CP2}, \cite{CPsmallmulti}, \cite{CPmixed}. 
In mixed characteristic, in particular the 
approaches of Lipman \cite{Lip} and Cossart-Jannsen-Saito 
\cite{CJS} exemplify two different ways to deal with the task: using
a combination of normalization and blow-up steps in the first case and relying
solely on blowing up at suitably chosen centers in the second. In Lipman's
approach the normalization in mixed characteristic proved to be a bottleneck 
in practice \cite{ResLIB}; the Cossart-Jannsen-Saito (CJS) 
algorithm, on the other hand, provides the structural advantage of only using 
a single kind of birational morphism.

The motivation behind our study of the CJS-algorithm is threefold: first, we
need a sufficiently powerful tool for desingularization in this case to 
enter a more structured, experiment-driven approach to the comparison of 
desingularization and valuation theoretic approach started in 
\cite{FKW}; second, key parts of the algorithm to determine the center
can also be used for an alternative approach to normalization; last, but not
least, the quest for a solution to the desingularization problem in higher
dimension makes it tempting to experimentally search for cases where a 
straight-forward generalization of an existing algorithm for surfaces to 
higher dimensions reveals new interesting phenomena. Section 6 of this article provides some more details
on the last two motivations.

\smallskip 

The goal of this article is to introduce a variant of the approach of 
Cossart, Jannsen and Saito to embedded desingularization in dimension $2$,
which is sufficiently explicit to allow implementation. The underlying ideas 
of the CJS-algorithm can be traced back to Hironaka's Bowdoin lectures 
\cite{HiroBowdoin} for hypersurfaces, for which Cossart later closed a gap in 
\cite{CosToho}. 
To further outline this approach in slightly more detail,
we first need to fix the setting and notation for embedded resolution in
mixed characteristic:  Let $ X $ be a reduced excellent 
Noetherian scheme of dimension two, embedded in some excellent regular 
scheme\footnote{In fact, if one is only interested in understanding the 
	local situation, it is always possible to reduce all consideration to the 
	embedded setting, see \cite{CJS} Remark \xch{5.22}.} 
$ Z $. Note that we do not fix a base field, explicitly allowing the case of a
scheme over a Dedekind ring. We may assume without loss of generality that 
$ Z $ is connected (and hence that $ Z $ is irreducible) since $ Z $ is 
regular and we may safely treat each component separately. Throughout the 
article, we fix 
\[
N := \dim (Z) .
\]
In contrast to CJS, we avoid the use of the Hilbert-Samuel function  
$ H_X ( x) $ (loc.~cit. Definition~1.28) as measure for the complexity 
of the singularity of $ X $ at $ x $ and pursue the idea of using the order
of $I(X)$ at $x$ instead. More precisely, we introduce a refined order 
$ \nut(x) $ (Definition~\ref{Def:nut}), for which the computation of the
maximal locus is not as expensive as for the Hilbert-Samuel function.
In Example \ref{Ex:diff} we show that this leads to a different resolution 
process. The idea of replacing the Hilbert-Samuel function by the order is
by no means new, as it can e.g., be found in \cite{BEV}. 
More
recently, also the multiplicity has been used as a replacement for the 
Hilbert-Samuel function in \cite{AnaOrlandoClay}, but we will not follow that 
train of thought here.

\smallskip 

The article is structured to first cover the theoretical background, before
formulating the algorithms first in a sequential way and then as Petri-nets.
More precisely, section 2 discusses our variant of the algorithm of 
Cossart-Jannsen-Saito without assuming that the reader is already deeply 
familiar with desingularization algorithms. This section therefore also recalls
important results and constructions from \cite{CJS} and \cite{InvDim2}, and 
explains the complete algorithm illustrating important facts by examples, 
wherever necessary. In particular, it introduces the main invariant, the
refined order of an ideal, and addresses the question of termination for the 
two simplest cases, namely the 1-dimensional case and the case of 2-dimensional 
hypersurfaces.

As the main measure for the complexity of a singularity is the order of 
the corresponding ideal in our setting, the computation of the locus of maximal
order is the very heart of the algorithm. As this is not a completely trivial
task in mixed characteristic, section 3 is devoted to treating this in detail.
This section does not assume that the reader is familiar with the notion of
Hasse{-Schmidt} derivatives, which are indispensable for the computation, and 
hence also discusses them briefly before rendering the construction of the 
locus of maximal refined order completely algorithmic.

Based on the theoretical discussions in sections 2 and 3, we are then prepared
to reformulate our version of the CJS-approach in terms of sequential 
algorithms in section 4.
In here, we apply particular care to the identification of bad
primes or more precisely potentially bad primes, which we refer to as 
interesting primes 
(i.e., prime numbers above which there possibly lies a component of the maximal locus). 
At this point, it is also important to note that we 
distinguish between horizontal and vertical components in the maximal locus.
The former are components that can be seen at every prime, while each vertical component is associated to a single particular prime $ p $ and is not seen above any prime $ q \neq p $.
In contrast to \cite{CJS}, we give precedence to the horizontal components 
and consider the vertical ones only when no more horizontal components are
left.%
\footnote{The idea to try to resolve first all horizontal components 
	and then to handle the vertical components has been pointed out by 
	H.~Hironaka during a private {conversation} at a conference in 2006 at the ICTP
	in Trieste, Italy.}

Given the sequential algorithms and the naturally parallel structure due to 
the use of coverings by charts at different points in the algorithms, we
rephrase the algorithms in terms of Petri-nets in section 5 to make them
accessible to the parallel workflow management system GPI-Space. As we also
do not assume familiarity of the reader with the formalism of Petri nets, we
give a very brief overview at the beginning of the section, before stating the
nets which represent the main parts of the algorithms.

The last section is then devoted to an outlook which details the possible
applications already mentioned as motivations for this article.

\smallskip

\noindent 
{\em Acknowledgments:}
The authors would like to thank Vincent Cossart, Gerhard Pfister, Mirko Rahn, 
and Stefan Wewers for fruitful exchange on topics touching upon the context of 
this article.
{Further, they thank Sabrina Gaube for numerous helpful comments on an earlier draft.}
\xch{\label{xch:thanks}They are grateful to the anonymous referees for useful comments and questions.}

\smallskip 

\noindent 
{\em Notational Remark:} 
throughout the article, $ \NN= \{ 0, 1, 2, \ldots \} $ and we use multi-index notation, i.e., for $ (\bx)  = (x_1, \ldots, x_m) $
and $ \ba = (a_1, \ldots, a_m) \in \NN^m $, we abbreviate 
$
\bx^\ba = x_1^{a_1} \cdots x_m^{a_m} . 
$
{We often apply abuse of notation and write
	$ V(J) := \Spec(A[\bx]/J) $, for an ideal $ J \subset A[\bx] $ and a principal ideal domain $ A $ (e.g., $ A = \ZZ $).}

\section{A variant of the algorithm of Cossart, Jannsen and Saito}

We present a variant of the algorithm by Cossart, Jannsen, Saito \cite{CJS}.
In contrast to their approach, we do not consider the Hilbert-Samuel function 
$ H_X : X \to \NN^\NN $
as measure for the complexity of the singularities, but a refinement 
$ \nut : X \to \NN^{2} $ of the order (Definition~\ref{Def:nut}).
The benefit is that with our modification the implementation of
the algorithm runs in reasonable time, whereas the computation of 
the Hilbert-Samuel locus is far more expensive.
(cf.~ \cite{RaschkeDipl}).
Furthermore, we make a distinction between the \textit{vertical} components 
(i.e., those lying completely in one fiber above some prime $ p \in \ZZ$)
in the maximal locus of $ \nut $ and the \textit{horizontal} ones 
(which appear in every fiber). 

Let us point out that in this section we do not make use of the 
additional structure imposed on $ Z $.
Hence the constructed resolution procedure can be applied for any
excellent irreducible regular scheme $ Z $.

From the viewpoint of resolution of singularities, a map 
$ \mu : X \to S $ into a totally\footnote{In fact, it is sufficient if the set is 
	partially ordered if we require that $ \mu $ can only achieve finitely 
	many maximal values on $ X $, cf.~\cite{CJS} Lemma \xch{2.34}(c) 
	and the preceding paragraph.
	For example, this appears if the Hilbert-Samuel function is used.} 
ordered set $ (S, \leq ) $
has to fulfill the following properties in order to be a decent 
measure for the complexity of a singularity and its improvement during the resolution process:

\begin{enumerate}
	\item[\bf (A)] $ \mu $ {\em distinguishes singular and regular points}.
	\item[\bf (B)] $ \mu $ is {\em (Zariski) upper semi-continuous}: 
	the sets 
	$ X_\mu (\geq s) := \{ x \in X \mid \mu(x) \geq s \} $ are closed,
	for every $ s \in S $. 
	\item[\bf (C)] $ \mu $ is {\em infinitesimally upper semi-continuous}:
	$ \mu $ does not increase %
	if we blow up a regular center $ D $ contained in the locus, 
	where $ \mu $ is maximal. 
	\item[\bf (D)] $ \mu $ can only {\em decrease finitely many times} until all points are regular.
\end{enumerate}

Therefore the order itself is not an appropriate measure since {(A)} 
fails to hold, in general, see Example~\ref{Ex:order_not_good}.
In our setting, we shall use $ S = \NN^{2} $ equipped with the lexicographical 
order. Hence, for us, {(C)} implies {(D)}.

Recall the following criterion for upper semi-continuity:

\begin{Lem}[\cite{CJS} Lemma \xch{2.34}(a)]
	\label{Lem:CJS_usc}
	A map $ \mu : X \to S $, with $ (S , \leq ) $ a totally ordered set, is upper semi-continuous if and only if the following holds:
	\begin{enumerate}
		\item 
		If $ x, y \in X $ with $ x \in \overline{ \{ y \} } $, 
		then $ \mu(x) \geq \mu(y) $.
		\item 
		For all $ y \in X $ there is a dense open subset
		$ U \subset \overline{ \{ y \}} $ 
		such that $ \mu(x) = \mu(y) $ for all $ x \in U $.
	\end{enumerate}
\end{Lem}
Since we assume that $ \mu $ cannot increase under suitable blowing ups and that $ \mu $ can only
decrease finitely many times until the strict transform of $ X $ 
is regular, it is sufficient to construct a finite sequence of blowing ups such that the maximal value achieved by $ \mu $ decreases strictly.  
This implies then embedded resolution of singularities for $ X \subset Z $.
More precisely, we need to answer the following question:
Given $ X \subset Z $, can we find a finite sequence of blowing ups,
\begin{equation}
\label{eq:resol_seq}
\begin{array}{rlcrlcccrl}
Z =: & Z_0 & \longleftarrow & \Bl_{D_0} (Z_0) =: & Z_1 & \longleftarrow & 
\cdots & \longleftarrow & \Bl_{D_{a-1}} (Z_{a-1}) =: &  Z_a \\[5pt]
&\bigcup	&	& & \bigcup	&	& & &  &		\bigcup	\\[5pt]
X =:  & X_0	& \longleftarrow	& & X_1	& \longleftarrow 	& \cdots & \longleftarrow  &  &		X_a	, 
\end{array}
\end{equation}
in centers $ D_i $ contained in the locus, where $ \mu $ is maximal on $ X_i $,
$ 0 \leq i \leq a -1$,
($ X_i $ denotes the strict transform of $ X $ in $ Z_i $) such that eventually
\[
\max \{ \mu(x_a) \mid x_a \in X_a \}
<
\max \{ \mu(x) \mid x \in X \}
?
\]

As a first step we introduce our refinement $ \nut $ of the order and 
then show that the above properties hold for $ \nut $
in order to justify its use.

\begin{Def}
	Let $ x \in  X $ (not necessarily closed).
	We denote by $ ( R = \cO_{Z,x}, \mathfrak{m}, k = R/\mathfrak{m} ) $ the local ring of $ Z $ at $ x $ 
	(which is excellent and regular) and
	by $ J \subset R $ the ideal which defines $ X $ locally at $ x $.
	\begin{enumerate}
		\item 
		The {\em order of $ X $ at $ x $} is defined as the 
		order of $ J $ at $ \mathfrak{m} $,
		\[
		\ord_x (X) := \ord_\mathfrak{m}(J) := \sup \{ t \in \NN \mid J \subset \mathfrak{m}^t \}
		\]
		For an element $ f \in J $, we also abbreviate 
		$ d(f) := \ord_\mathfrak{m}(f) $.
		\item 
		The {\em maximal order of $ X $} is defined as
		\[
				\maxOrd(X) := \sup\{ \ord_x(X) \mid x \in X \}
		\]
		and the {\em maximal order locus of $ X $} is the locus of points of order $ \maxOrd(X) $,
		\[
		\MaxOrd(X) := \{ x \in X \mid \ord_x(X) = \maxOrd(X) \}.
		\]
		\item 
		The {\em initial form of $ f $ (with respect to $ \mathfrak{m} $)} is defined as
		\[
		\ini_\mathfrak{m}(f) := f \mod \mathfrak{m}^{d(f) +1} \in \gr_\mathfrak{m} (R),
		\]
		where $ \gr_\mathfrak{m}(R) = \bigoplus_{t \geq 0} \mathfrak{m}^t/\mathfrak{m}^{t+1} $ denotes the 
		{\em associated graded ring of $ R $ at $ \mathfrak{m} $}.
		\item 
		The {\em initial ideal of $ J $ at $ \mathfrak{m} $} is defined as the 
		ideal $  \Ini_\mathfrak{m} (J) $ in $ \gr_\mathfrak{m}(R) $ generated by the initial forms 
		of the elements in $ J $,
		\[
		\Ini_\mathfrak{m} (J) := \langle \ini_\mathfrak{m}(f) \mid f \in J \rangle.
		\]		
	\end{enumerate}
\end{Def}

Recall that 
the order is a upper semi-continuous function, 
see \cite{Hiro64} Chapter III \S 3 Corollary~1 p.~220.
Furthermore, the graded ring $ \gr_\mathfrak{m}(R) $ is isomorphic to a polynomial 
ring in $ n $ variables over the residue field $ k $
(even in mixed characteristic), where  
\[ 
n := n_x := \dim (R) .
\]
In the graded ring one of the variables may correspond to a prime number, e.g.~$ R = \ZZ[x]_{\langle 73, x \rangle} $. 
Moreover, $ \Ini_\mathfrak{m}(J) $ is a homogeneous ideal in $ \gr_\mathfrak{m}(R) $.
For $ \fM := \Ini_\mathfrak{m} (\mathfrak{m}) \subset \gr_\mathfrak{m} ( R) $, we have
\begin{equation}
\label{eq:order_graded_elementwise}
\min \{ \ord_\mathfrak{m}(f) \mid f \in J \} = \ord_\mathfrak{m} (J) 
= \ord_{\fM} ( \Ini_\mathfrak{m}(J) ). 
\end{equation}

\begin{Obs}
	\label{Obs:d_x}
	Let the situation be as in the previous definition.
	Set $ I_1 :=  \Ini_\mathfrak{m}(J) $.
	If $  \ord_{\fM} ( I_1 ) = 1 $, we can consider its image 
	$\overline{I_1}$ in the degree $1$ slice of $ {\gr}_{\mathfrak m}(R)$, 
	which yields a subspace of the finite dimensional $k$-vector space
	$ {\gr}_{\mathfrak m}(R)_1 =\mathfrak m / \mathfrak m^2 $ . 
	We can thus find a basis, say $F_1,\dots,F_a$ of $\overline{I_1}$,
	for some $ a := a_x \in \NN $. 
	Either $I_1 = \langle F_1,\dots,F_a\rangle_{gr_{\mathfrak m}(R)}$ 
	or $\ord_{\fM} (H) > 1$ for all elements $H$ in the ideal $ I_{a+1} $ generated by the set 
	$I_1 \setminus \langle F_1,\dots,F_a\rangle_{gr_{\mathfrak m}(R)}$. 
	We define
	\begin{equation}
	\label{eq:d_x} 
	d_x := 
	\left \{
	\begin{array}{cl}
	1, 	& \mbox{ if }  I_1 = \langle F_1, \ldots F_a \rangle \\[5pt]
	\ord_\fM ( I_{a+1} ) > 1,	& \mbox{ otherwise.}
	\end{array}
	\right.
	\end{equation}
	
	For each $ F_i \in \Ini_\mathfrak{m}(J) $, we can choose a lift
	$ f_i \in J $ with $ \ini_\mathfrak{m}(f_i) = F_i $, $ 1 \leq i \leq a $.
	We have $ \ord_\mathfrak{m} (f_i) = \ord_\fM (F_i) = 1 $ and
	$ (f_1,\dots,f_a) $ forms a regular sequence, 
	\xch{\label{xch:i.e}in particular},
	$ f_i \notin \langle f_1, \ldots, f_{i-1} \rangle $,
	for $ 1 \leq i \leq a $. 
	Hence $ \mathcal Y := V(f_1, \ldots f_a ) $ defines a 
	$ ( n- a )$-dimensional regular	{subscheme} of 
	$ \mathcal Z := \Spec (R) $ which contains 
	$ \mathcal X = V(J) $.
	This implies 
	\[ 
	a_x \leq n_x - \dim \mathcal X \leq N = \dim (Z).
	\] 
	The following are equivalent:
	\begin{enumerate}
		\item[(i)] 	the first inequality 
		is strict, $ a_x < n_x - \dim \mathcal X $ (i.e., $ \dim \mathcal X < \dim \mathcal Y $),
		\item[(ii)] 	$ d_x > 1 $, and
		\item[(iii)] \xch{$ X $ is not regular at $ x $.} \label{xch:not_regular_instead_singular}
	\end{enumerate}
	
	At a closed point corresponding to a maximal ideal ${\mathfrak m}_x$ the value of
	$ n_x - a_x $ coincides with the embedding dimension of $ X $ at $ x $,
	\[
	n_x - a_x = \embdim_x ( X ) .	
	\]
	However, this idea needs to be refined slightly to also allow local considerations at 
	non-maximal prime ideals. This is taken into account in the following definition. 
	The reasons for the change will become apparent in Example \ref{n-aNotEmbdim}
	directly after it.
\end{Obs}

\begin{Def}
	\label{Def:nut}
	Let $ X \subset Z $ be as above.
	\begin{enumerate}
		\item Using the previous notations, we define the map
		\[ 
		\nut := \nut_{X,Z}:  X  \to  ( \NN^2, \leq_{lex}), \ \ \ \ \ \
		\nut(x) :=  (\,N-a_x,\, d_x\,)
		\xch{\label{xch:point}.}
		\]
		We call $ \nut $ the {\em refined order on $ X $}.
		\item For $ A \in \NN^2 $, we set
		$
		\Ord{\geq A}(X) := \{ \,x \in X \mid \nut(x)  \geq A \, \}
		$.  
		We define the {\em maximal refined order of $ X $} by
		\[
		\maxNu(X) : = (\alpha, \delta ) := \max \{ \,\nut(x) \mid x \in X \, \}
		\]
		and the {\em maximal refined order locus of $ X $} by
		\[ 
		\MaxNu ( X ) :=
		\left\{ 
		\begin{array}{cl}
		\Ord{\geq \maxNu(X)}(X) ,
		& \mbox{if } \delta > 1 ,\\[3pt]
		X, & \mbox{if } \delta = 1.
		\end{array}
		\right.  
		\]
		\item 	A closed subscheme $ D $ 
		(not necessarily irreducible)
		is called a {\em \good center} if it is regular and contained in the maximal refined order locus of $ X $. 
		We call the blowing up with such a center $ D $ a 
		{\em \good blowing up}.
	\end{enumerate}
\end{Def}

As stated above, $ \nut $ depends on $ X $ and $ Z $.
In order to keep formulas simple, we will only mention the 
dependence when it is not clear from the context.

The following example illustrates, why we are using $ N - a $ and not $ n - a $ in the definition of $ \nut $. The difference is subtle, but crucial for further reasoning.

\begin{Ex}[Whitney umbrella] \label{n-aNotEmbdim}
	Let $ k $ be an infinite field and 
	$ X \subset Z := \mathbb{A}^3_k = \Spec (k  [t,y,z] ) $ be the
	hypersurface defined by the polynomial $ f = t^2 - y^2 z $.
	Its singular locus is the curve $ C := V ( t,y ) $, whose generic point
	we denote by $ \eta$. By $ x $ we denote the closed point $ V ( t,y,z) $.
	
	We have $ a_x = a_\eta = 0  $ and $ d_x = d_\eta = 2 $ coincides
	with the order of $ f $ at the respective point.
	Therefore $ \nut ( x) = \nut (\eta ) = (3,2) $
	since $ N = \dim ( Z ) = 3 $,  
	and $ \MaxNu(X) = V(t,y) $.
	
	On the other hand, if we used $\nu_{hyp}= (n_x - a_x, d_x) $ instead of $\nut$
	we would see the following: $ n_x = \dim (\mathcal{O}_{Z,x} ) = 3 $
	and $ n_\eta = \dim (\mathcal{O}_{Z,\eta} ) = 2  $, since $z$ becomes an
	element of the residue field at $\eta$ and we see that
	$\operatorname{trdeg}((\mathcal{O}_{Z,\eta}/\mathfrak{m}_{Z,\eta}) : 
	(\mathcal{O}_{Z,x}/\mathfrak{m}_{Z,x})) = 1$ .
	Hence we have 
	$ n_x - a_x = 3 > 2 = n_\eta - a_\eta $, i.e., $\nu_{hyp}(x) =(3,2) >
	(2,2) = \nu_{hyp}(\eta)$.
	In fact, this computation holds for any closed point $ x' $ 
	contained in $ C $ instead of $ x $. 
	Since $ k $ is infinite, there are infinitely many closed points on $ C $
	and the maximal locus of $ \nu_{hyp} $ is not closed. 
	
	Another interesting aspect about this example is the following:
	if we blow up the closed point $ x $ then the strict transform
	of $ f $ 
	at the point with parameters $ (t',y',z') = (\frac{t}{z}, \frac{y}{z}, z)$
	is given by $ f' = t'^2 - {y'}^2 z' $ and we have thus encountered a loop.
\end{Ex}

\begin{Rk} \phantomsection\label{rem_nut_is_nu_trunc}
	\begin{enumerate}
		\item 	The reader familiar with the notion of Hironaka's 
		$ \nu^*$-invariant (\cite{Hiro64} Chapter III Definition 1, p.~205, see also \cite{CJS} Definition \xch{2.17}) will recognize a connection.
		Observation~\ref{Obs:d_x} is precisely the beginning 
		of a possible way to construct a so called standard basis for $ J $ 
		and thus to determine $ \nu^*(J, R) $, see \cite{CJS} the paragraph before Remark \xch{2.5}.
		Set $ I := \Ini_\mathfrak{m} (J) $, then
		\[ 
		\begin{array}{rcl} 
		\nu^*(J, R) & =  
		& (\nu^1 (I), \nu^2 (I), \ldots , \nu^a (I), \nu^{a+1} (I), \nu^{a+2}(I), \ldots ) = \\[5pt]
		& =  
		& (1, 1, \ldots, 1, d_x, \nu^{a+2}(I), \ldots ) 
		\in (\NN \cup \{\infty \})^\NN.
		\end{array}
		\]
		\label{xch:meaning_nu_sta}
		\xch{Note that our invariant at the closed point $ x $ corresponding to the maximal ideal of $ J $ is $ (N - a_x, d_x) $, where $ a_x = a $.
			Hence, our invariant corresponds to the truncation of the $ \nu^*$-invariant after the first entry $ > 1 $. 
			
			The entries $ \nu^1 (I), \nu^2 (I) , \ldots $ are ordered increasingly and are equal to the orders of a particular
			system of generators $ f_1, f_2 , \ldots $ for $ J $, whose behavior along a blowing up is the same as the one of the ideal $ J $.
			By \cite{CJS} Theorem 3.10, the Hilbert-Samuel function decreases after a permissible blowing up if and only if the $ \nu^* $-invariant does.}

		\item 	If $ X = V(f) $ is an affine hypersurface then $ \nut $ determines
		the order of $ X $ and vice versa.  
		On the other hand, the behavior of the Hilbert-Samuel function is also 
		detected by the order in this case (\cite{CJS} Theorem \xch{3.10}).
		Hence the maximal loci considered by the original CJS-algorithm 
		and by our variant coincide
		for hypersurfaces.
		Nonetheless, the algorithms may differ in their way of resolving the maximal locus since this is in general not a hypersurface.
		\item 	Since we assume $ \dim (X) = 2 $ and $ X $ reduced, $ \MaxNu(X) $ has at most dimension one.
		In particular, $ \MaxNu(X) $ itself has at most isolated singularities.		
	\end{enumerate}
	
\end{Rk}

\begin{Ex}
	\label{Ex:order_not_good}
	Let $ R = \ZZ[t,v,w,y,z]_{\mathfrak{m} } $, where $ \mathfrak{m} = \langle p, t,v,w,y,z\rangle $,
	for $p \in \ZZ $ prime.
	We consider the {scheme} $ X := V(J) \subset \Spec (R) $, where 
	\[ 
	J := \langle \, p, \, t, \, v^2 - y^3, \, z^5  - y^2 w^5  \, \rangle \subset R .
	\]
	Then, $ \ord_\mathfrak{m} (J) = 1 $, although the point $ x $ in $ X $ corresponding 
	to $ \mathfrak{m} $ is singular.
	Hence, the order cannot distinguish regular and singular points. 
	
	On the other hand, since $ \ord_\mathfrak{m} (v^2 - y^3) = 2 $ is maximal 
	and $ \ord_\mathfrak{m} (z^5  - y^2 w^5) = 5 \geq 2  $,
	we have $ \nut(x) = (4,2) $.
	In fact, $ \maxNu(X) = (4,2) $ and $ \MaxNu ( X ) = V(p,t,v,y,z) $. 
	
	Note: if we put $ P := \langle p,t,v,y,z \rangle $, then
	$
	\ord_P (z^5  - y^2 w^5) = 2 < 5 = \ord_\mathfrak{m} (z^5  - y^2 w^5),
	$  
	a difference which is undetected by $ \nut $.
\end{Ex}

\begin{Prop}
	\label{Prop:A-D_for_nut}
	The map $ \nut = \nut_{X,Z} : X \to \NN^2 $ 
	distinguishes regular from singular points {(A)}, 
	is Zariski upper semi-continuous {(B)}, 
	is infinitesimally upper semi-continuous {(C)}, 
	and can only decrease finitely many times until it detects regularity {(D)}.
\end{Prop}

\begin{proof}
	Since $ \nut $ takes values in $ \NN^2 $, {(D)} is implied by {(C)}.
	Further, we have seen in Observation \ref{Obs:d_x} 
	that $ d_x $ distinguishes singular and regular points.
	Hence {(A)} holds.
	
	For {(C)}, we may assume that $ X $ is singular.
	Let $ x \in \MaxNu(X) $.
	As before, $ J $ denotes the ideal in $ R:= \cO_{Z,x} $
	defining $ X $ locally at $ x $.
	Let $ f_1, \ldots, f_a, \ldots, f_r \in J $ be generators for $ J $
	such that $ \ord_\mathfrak{m} (f_i) = 1  $ and 
	$ f_i \notin \langle f_1, \ldots f_{i-1} \rangle $,
	for $ 1 \leq i \leq a $, 
	and $ \ord_\mathfrak{m} (f_{j}) > 1  $, for $ j > a $
	(cf.~Observation~\ref{Obs:d_x}).
	In particular, $ V(f_1, \ldots, f_a ) $ is regular.
	
	Let $ \pi : Z' := \Bl_D (Z) \to Z $ be a \good blowing up with center 
	$ D \subset \MaxNu(X) $ containing $ x $.
	Denote by $ X' $ the strict transform of $ X $ in $ Z' $.
	We have to show $ \nut (x') \leq \nut (x) $, 
	for every $ x' \in \pi^{-1} (x) $. 
	(Note that $ \dim (Z') = \dim (Z) $).
	
	Above the point $ x $, the strict transform of $ X $ is given by the 
	strict transforms of the elements $ h \in J $. 
	Since $ X $ is singular, we can choose $ h \in J $ such that 
	$ \ord_{\overline{\mathfrak m}} ( \overline{h} ) = d_x $, where 
	$\overline{\mathfrak m}$ and $\overline{h}$ denote the images of the
	respective objects in $R/\langle f_1, \ldots f_a \rangle$.
	\label{xch:Mora} Without loss of generality, we may assume that $ h $ is of order $ d_x $ at $ x $.
	\footnote{\xch{The representative can, for instance, be chosen as the 
        element $f_{a+1}$ in the above mentioned computation of a standard 
        basis leading to the $\nu^*$-invariant, cf.  
        Remark~\ref{rem_nut_is_nu_trunc}.\\
        From another perspective and moving $x$ to the origin, the knowledge of
        $ \ord_{\overline{\mathfrak m}} ( \overline{h} ) = d_x $
        implies that after translation the Mora normal form 
        (cf. \cite{GP}, Algorithm 1.7.6, for a textbook reference) 
        of an arbitrary representative of $\overline{h}$ w.r.t a standard 
        basis of $\langle f_1,\dots,f_a \rangle$ will be of order $d_x$. 
        Of course, retranslating the resulting normal form back to $x$, we
        obtain the desired representative.}}

	Note that $ D $ is a weakly permissible center for the hypersurface 
	$ V(h) \subset Z $ at least in a Zariski-open neighbourhood of $x$
	(i.e., it is regular and contained\footnote{It may, of course happen
		that the maximal refined order of $h$ exceeds $d_x$ somewhere else, but 
		this can only happen on a Zariski-closed subset not meeting $x$.
		Hence it is possible to avoid this set by taking a suitable
		Zariski-open neighbourhood of $x$.} in $ \MaxNu( V(h)) $ 
	-- in fact the latter coincides with the maximal order locus of $ V(h) $).
	It is a well-known fact (see \cite{CJS} Theorem~{\xch 3.10}, using loc.~cit.~Definitions~{\xch 2.26}
	{and} \xch{2.17}(2)) 
	that  
	the order of a hypersurface does not increase under such blowing ups.
	This implies $ \nut (x') \leq \nut (x) $  and hence  {(C)}.

	It remains to prove {(B)}, which is the upper semi-continuity
	of $\nut$. 
	For this, we will make use of Lemma \ref{Lem:CJS_usc}:
	Let $ x, y \in X \subset Z $ with $ x \in \overline{ \{ y \} } $.
	We have to show that the inequality $ \nut(x) \geq \nut(y) $ holds. 
	Let $ U_0 = \Spec(S) $ be an affine open set containing $ x $ and $ y $.
	Let $ f_1, \ldots, f_a \in S $ be such that
	each $ f_i $ is of order one at $ x $
	and they define $ a_x = a $ 
	(cf. Observation \ref{Obs:d_x}).
	By the upper semi-continuity of the order, 
	the order of each $ f_i $ does not increase by passing from $ x $ to $ y $.
	Since $ y \in X $, the order of $ f_i $ has to be at least one at $ y $.
	Hence, we have $ a_y \geq a_x =a $ and if the inequality is strict then $ \nut (y) < \nut (x) $.
	
	Suppose $ a_y = a_x $. 
	We stay in the same notation ($ J \subset S $ the ideal defining $ X \cap U_0 $ in $ U_0  = \Spec (S) $ containing $ x $ and $ y $).
	Let $ h \in J $ be an element of order $ d_x $ at $ x $ 
	(cf.~proof of (C)).
	By choosing $ U_0 $ sufficiently small, we may assume that 
	$  V(f_1, \ldots, f_a ) \subset U_0 $ is regular.
	Let us denote by $ \overline{h} $ the image of $ h $ in $ S/\langle f_1 , \ldots, f_a \rangle $.
	Note that $ x, y \in V(f_1, \ldots, f_a ) $.
	We have that the order of $ \overline{h} $ at $ x $ is $ d_x $.
	Again, by the upper semi-continuity of the order, we obtain that the order of $ \overline{h} $ at $ y $ is at most $ d_x $.
	Therefore, $ d_y \leq d_x $, i.e., 
	$ \nut (y) \leq  \nut (x) $.
	In other words, Lemma \ref{Lem:CJS_usc}(1) holds.

	For Lemma \ref{Lem:CJS_usc}(2), let $ y \in X $.
	We need to show that there exists a dense open subset
	$ U \subset \overline{ \{ y \}} $ 
	such that $ \nut(x) = \nut(y) $ for all $ x \in U $.
	Let $ U_0 = \Spec(S) $ be an affine neighborhood of $ y $ in $ Z $ and denote by $ J \subset S $ the ideal defining $ X $.
	Using Observation \ref{Obs:d_x} at $y$ and analogous notation to
	before, we choose generators 
	$ ( f_1, \ldots, f_a, f_{a+1}, \ldots, f_r ) $ for $ J $  in $ S $ 
	such that $ V (f_1, \ldots, f_a ) $ is regular and 
	$N-a$ is the first entry of $\nut(y)$.
	(Note: We can choose $ U_0 $ sufficiently small for this).
	From the construction of $ d_y $ it follows that $ d_y $ coincides
	with the order of $ J \cdot S/ \langle f_1, \ldots, f_a \rangle $ at $ y $.
	The latter is upper semi-continuous and by Lemma \ref{Lem:CJS_usc} there exists 
	a dense open subset $ V $ contained in $ V (f_1, \ldots, f_a ) $
	such that the order at every point in $ V $ is equal to $ d_y $.
	This yields then a dense open set $ U \subset \overline{ \{ y\} }$ such that $ \nut(x) = \nut(y) $ for all $ x \in U $. 
	Hence Lemma~\ref{Lem:CJS_usc}(2) holds.
	As a conclusion we obtain that $ \nut $ is upper semi-continuous.
\end{proof}

The resolution process provides us with more data than just the strict 
transform $ X' $ of $ X $ in $ Z '$.
Namely, there are also the exceptional divisors $ E' $ of the blowing ups.
Additionally to the above, 
one imposes for an embedded resolution of singularities that the strict
transform of $ X $ has at most simple normal crossings with $ E' $ and 
that all centers in \eqref{eq:resol_seq} meet the exceptional divisors 
transversally.

Furthermore, in order to attain a canonical resolution of singularities one has 
to take the exceptional divisors into account.
For example, the singular locus of $ X := V ( x^2 - y^2 z^2 ) $ (which coincides
with $ \MaxNu(X) $) consists of the two lines $  L_1 := V(x,y) $ 
and $ L_2 := V(x,z) $.
A priori there is no way to distinguish them and we need to blow up their
intersection $ V(x,y,z) $.
At first glance the situation did not change locally at the point with
coordinates $ ( x', y', z') = ( \frac{x}{z}, \frac{y}{z}, z ) $.
But, in fact, $ V(x', z') $ is contained in the exceptional divisor, 
whereas $ V(x',y') $ is the strict transform of $ L_1 $.
Thus we can distinguish the lines and pick one of them as the next center.
(For example, the CJS-algorithm will choose $ L_1 '$  since it is "older" than
the other component,
see also Example \ref{Ex:loop}). 

It is a general philosophy in resolution of singularities to encode the 
information on the exceptional divisors in some way. 
Following \cite{CJS}, we do this using a so called boundary $ \cB $ (\cite{CJS} Definition \xch{4.3}) 
and adapt the notion of $ \cB $-permissible centers 
(loc.~cit. Definition \xch{4.5}) for our situation.  

\begin{Def}
	Let $ X \subset Z $ be as before.
	\begin{enumerate}
		\item 	A {\em boundary on $ Z $} is a collection 
		$ \cB= \{ B_1, \ldots, B_d \} $ 
		of regular divisors on $ Z $ such that the associated 
		divisor $ \mathrm{div}(\cB) = B_1 \cup \cdots \cup B_d $ 
		is a simple normal crossings divisor.
		For $ x \in Z $, we let $ \cB(x) := \{ B_i \mid x \in B_i \} $ be the boundary components passing through $ x $.
		
		\item A closed subscheme $ D \subset X $ is called {\em $ \cB $-\good}if it is \good and has at most simple normal crossings with the boundary $ \cB $
		(i.e., $ D $ is regular, contained in $ \MaxNu(X) $, and intersects $ \mathrm{div}(\cB) $ transversally).
		
		We also say that the corresponding blowing up 
		$ \pi : Z' :=\Bl_D(Z) \to Z $ with center $ D $ is a 
		{\em $ \cB$-\good blowing up}.
		The {\em transform $ \cB' $ of the boundary $ \cB $ under $ \pi $} 
		is defined to be the union of the strict transforms of the 
		components of $ \cB $ and the exceptional divisor $ E'_\pi $ of $ \pi $
		(i.e., $ \cB' = \{ B_1' , \ldots, B_d', E'_\pi \} $ with the obvious notations).
	\end{enumerate}
\end{Def}

We {\em distinguish the boundary components} in two sets
(cf.~\cite{CJS} Lemmas \xch{4.7} and \xch{4.14}):
Suppose we are at the beginning of the resolution process
and we consider some point $ x $.
Then we define all boundary components passing through $ x $ 
to be {\em old}, $ O(x) := \cB(x) $. 

Let $ \pi :Z' \to Z $ be a blowing up of $ Z $ in a
regular center $ D \subset X $ containing $ x $.
Let $ x' \in X' $ be a point on the strict transform of $ X $
which lies above $ x $, $ \pi(x') = x $.

\begin{itemize}
	\item 
	If $ \nut (x') < \nut(x) $ then the singularity improved at $ x' $
	and we treat the situation as if we are at the beginning.
	Hence $ O(x') := \cB'(x') $.
	
	\item 
	Suppose $ \nut (x') = \nut(x) $. 
	Then we define the old boundary components 
	$ O(x') $ to be the subset of $ \cB'(x') $
	consisting of the strict transforms of the elements 
	in $ O(x) $ which pass thorough $ x' $.
	All other components are called {\em new}, $ N(x') := \cB'(x')  \setminus O(x') $.
\end{itemize}
The reason for this distinction is that we have some control on 
the new boundary components and how they behave with respect to $ X $. 
On the other hand, we loose this control once the singularity 
strictly improved, i.e., once $ \nut $ decreases strictly.

\begin{Def} \label{DefLogRefinedOrder}
	Let $ X \subset Z $ be as before and let $ \cB $ be a boundary on $ Z $.
	We define the {\em log-refined order} by 
	\[ 
	\begin{array}{cccl}
	\Onut = \Onut_{X,Z} : & X & \longrightarrow 	  & ( \NN^3, \leq_{lex} ) \\[3pt]
	& x & \mapsto & (\nut(x), |O(x)|). 
	\end{array}
	\]
	For $ \widetilde{A} \in \NN^3 $, we set
	\[
	\begin{array}{ccc} 
	\OOrd{\geq \widetilde{A}}(X) := 
	\{ x \in X \mid \Onut(x)  \geq \widetilde{A} \},
	%
	\end{array}	
	\]
	\[	%
	\OmaxNu(X)  :=  (\alpha, \delta , \sigma ) := 
	\max \{ \Onut(x) \mid x \in X  \},
	\]	
	\[	
	\OMaxNu ( X ) :=
	\left\{ 
	\begin{array}{cl}
	\OOrd{\geq \OmaxNu(X)}(X), 
	& \mbox{if } \delta > 1 \mbox{ or } \sigma >0 ,\\
	X, & \mbox{if } \delta = 1 \mbox{ and } \sigma = 0 .\\
	\end{array}
	\right.		
	\]
\end{Def}

Let us explain how to obtain $ \OOrd{\geq (A,b)} (X) $ from $ \Ord{\geq A} (X) $, for $ A \in \NN^2 , b \in \NN $:

\begin{Construction}
	\label{Constr:vonMaxnachMaxO}
	Let the situation be as in the previous definition and, in particular, let 
	$ \cB = \{ B_1, \ldots, B_ d \} $.
	Let $ \widetilde A = (A, b) \in \NN^3 $ with $ A \in \NN^2 $, $ b \in \NN$.
	{First, we observe that, 
		for any $ A' \in \NN^2 $ with $ A' > A $, we have $ (A', 0) > \widetilde A $.
		Hence, $ \bigcup_{A'>A} \Ord{\geq A'} (X) \subseteq \OOrd{\geq \widetilde{A}} (X) $.
		The upper semi-continuity 
		of $ \nut $ (Proposition~\ref{Prop:A-D_for_nut}), implies that 
		$
		\Ord{\geq A'}(X) =  
		\{ x \in X \mid \nut(x)  \geq {A'} \}
		$ 
		is closed.
		Since there are only finitely many $ A' $ with $ \Ord{\geq A'} (X) \neq \emptyset $,
		the union $ \bigcup_{A'>A} \Ord{\geq A'} (X) $ is closed.}
		{If}
		$ b > d $ or $ b > N = \dim(Z) $, 
		then $ \OOrd{\geq \widetilde{A}} (X) = {\bigcup_{A'>A} \Ord{\geq A'} (X)} $.
	
	Hence, let us assume that {$ b \leq \min\{d, N\} $}.
	We {introduce} the closed set
	\[ 
	Y_0 := \bigcup_{(i_1, \ldots, i_b)} ( \Ord{\geq A}(X) \cap B_{i_1} \cap \ldots \cap B_{i_b} ) ,
	\]
	{where the union runs through all subsets of $ \{ 1, \ldots, d \} $ with $ b $ pairwise different elements.}
	Suppose there is an irreducible component $ W $ of $ Y_0 $
	{and a point $ w\in W $ such that $ \Onut(w) < \widetilde A $}.
	\xch{\label{xch:Constr2-11}(This may happen since we are intersecting with $ b $ exceptional components in the definition of $ Y_0$, among which some could be new.)}
	{If there exists at least one $ j \in \{1 ,\ldots, d\} $ such that $ W \nsubseteq B_j $,}
	then we replace $ W $ in $ Y_0 $ by 
	\[ 
	W_0 := \bigcup_{ j \,:\, W \nsubseteq B_j}  ( W \cap B_j).
	\]
	If we have $ W \subset B_j $, for {$1 \leq j \leq d $}, then
	{we replace $ W $ by the empty set $ W_0 := \emptyset $.} 
	The dimension of each irreducible component of $ W_0 $ is strictly smaller than $ \dim(W) $ if $ W_0 \neq \emptyset $.
	Hence, after finitely many iterations this procedure ends.
	Let us denote by $ Y_\ast $ the resulting closed set.
	By construction, we have that
	\[
	{\OOrd{\geq \widetilde{A}}(X) = Y_\ast \cup \bigcup_{A'>A} \Ord{\geq A'} (X)}.
	\]
	In particular, we obtain that $ \Onut $ is an upper semi-continuous function. 
\end{Construction} 

By combining the latter with Proposition \ref{Prop:A-D_for_nut}, we get
\begin{Prop}
	\label{Prop:A-D_for_nut_O}
	The map $ \Onut : X \to \NN^3 $ 
	distinguishes regular from singular points {(A)}, 
	is Zariski upper semi-continuous {(B)}, 
	is infinitesimally upper semi-continuous {(C)}, 
	and can only decrease finitely many times until it detects regularity {(D)}.
\end{Prop}

\begin{proof}
	By Proposition \ref{Prop:A-D_for_nut}, {(A)} holds.
	The previous construction implies property (B) and, 
	since $ \Onut$ takes values in $ \NN^3 $, {(D)} follows.
	It remains to show (C).
	
	A $ \cB $-\good blowing up is in particular \good and thus 
	$ \nut $ does not increase at points $ x' $ lying above $ x $.
	Furthermore, if $ \nut(x') = \nut(x) $ then the definition of $ O(x') $ implies
	$ |O(x')| \leq |O(x)| $.
	Therefore we obtain {(C)}.
\end{proof}

\begin{Def}
	\label{Def:horizon_vertical}
	Let $ X \subset Z $ be as above. 
	We distinguish the irreducible components of $ \OMaxNu (X) $ as follows:
	If $ Z $ contains a field, then all irreducible components are defined to be \textit{vertical}. 
	
	Suppose $ Z $ does not contain a field.
	Let $ Y \subset  \OMaxNu (X) $ be an irreducible component. 
	We say $ Y $ is a {\em horizontal component} of $ \OMaxNu (X) $ 
	if $ Y \times_{\Spec(\ZZ)} \Spec{(\QQ)} \neq \emptyset $.
	Otherwise, $ Y $ is called a {\em vertical component}.
	We denote by $\OMaxNu_{hor}(X) $ the set of horizontal components 
	of $ \OMaxNu (X) $.
\end{Def}

\begin{Ex}
	Consider $ Z = \Spec( \ZZ[x,y]) $
	with empty boundary.
	Let 
	\[ 
	X = V (x^2 - y^3 5^7 (y - 2)^4) .
	\]
	We leave it as an exercise to the reader to show that 
	\[ 
	\OMaxNu(X) = \MaxNu(X) = \MaxOrd(X) =  V(x, y) \cup V(x, 5) \cup V(x, y - 2) .
	\]
	By the previous definition, 
	$ V (x,y) $ is a horizontal component
	and $ V(x,5 ) $ is a vertical component lying above the prime $ 5 $.
	Furthermore, the third component, $ V (x, y - 2) $, 
	is also a horizontal component of $ \OMaxNu(X) $ since it is 
	non-empty after passing to $ \Spec (\QQ[x,y]) $.
	Note: $ V (x, y+p-2, p) \in V (x, y - 2) $ 
	since $ y - 2 = (y+p-2) - p $, for every prime $ p \in \ZZ$.
\end{Ex}

\begin{Def}
	Let $ \pi : Z' := \Bl_D(Z) \to Z $ be a blowing up with some irreducible regular center $ D $.
	Let $ Y' \subset Z' $ be an irreducible subscheme contained in the exceptional 
	divisor $ E' := \pi^{-1} (D) $, $ Y' \subset E' $.
	Then {\em $ Y' $ dominates $ D $} if its image under $ \pi $ 
	is dense in $ D $.
	
	If the center $ D $ is not necessarily irreducible, we say {\em $ Y' $ dominates $ D $} if there exists an irreducible component $ D_0 $ of $ D $ which is dominated by $ Y' $.
\end{Def}

For example, let $ Z = \mathbb{A}^3_S $ with coordinates $ (x,y,z) $, 
$ S $ any base, $ D = D_0 = V (x,y) $.
We consider the chart with coordinates $ ( x',y',z ) := (x,  \frac{y}{x}, z ) $.
In this local situation, the exceptional divisor $ E' $ is $ V (x') $, and
the component $ V(x', {y'} ) \subset E' $ dominates the center, 
whereas $ V(x', z) \subset E' $ does not, the
closure of its image in $ Z $ is $ V(x,y,z) $.

\smallskip

We are now prepared to formulate our variant of the {CJS}-algorithm.
We deliberately formulate this construction in a way close to the original one 
\cite{CJS} Remark \xch{6.29}.

\begin{Construction}
	\label{CJS-Constr}
	Let $ X $ be a reduced excellent Noetherian scheme 
	of dimension two, embedded in a regular scheme $ Z $,  $ X \subset Z $.
	Let $ \cB $ be a boundary on $ Z $.
	
	\smallskip 
	
	\noindent 
	\underline{\em Horizontal case:} \ Suppose $  \OMaxNu_{hor} (X) \neq \emptyset $, we set
	\[
	Y_0 := 
	\OMaxNu_{hor} (X)\xch{\label{xch:period2}.}
	\]
	Since $ X $ is reduced, $ \dim ( Y_0 ) < \dim ( X ) $
	unless $ X $ is regular in which case the construction of a desingularization is complete.
	By the induction on dimension we find a finite sequence 
	of $ \cB $-\good blowing ups,
	\begin{equation}
	\label{eq:inductionblowups}
	\begin{array}{ccccc}
	Z =:Z_0 &	\longleftarrow	& \cdots &	\longleftarrow 	& Z_{m_0} =: Z' 	\\[5pt]
	\cB=: \cB_0 	&		& \cdots &	 	& \cB_{m_0} =: \cB' ,
	\end{array}
	\end{equation}		
	such that the strict transform $ Y_0'$ of $ Y_0 $ in $ Z' $ is 
	a $ \cB' $-\good center for $ X_{m_0} $.
	The next center chosen according to the construction is 
	$ D_{m_0} := Y'_0 $,
	\[
	\begin{array}{ccr}
	Z' = Z_{m_0} 		&	\longleftarrow	& \Bl_{D_{m_0}} ( Z_{m_0} ) =: Z_{m_0 + 1} \,	\\[5pt]
	\cB' = \cB_{m_0} 	&			 	& \cB_{m_0 + 1} .
	\end{array}
	\]
	Let $ X_{m_0 +1} $ be the strict transform of $ X_0 := X $ under 
	the previous sequence of blowing ups.
	
	If $ \OmaxNu( X_{m_0 +1} ) < \OmaxNu(X_0) $ then 
	the singularity improved strictly 
	and the construction restarts with the new maximal value for $ \Onut $ and the 
	corresponding data.
	Further, if $ \OmaxNu( X_{m_0 +1} ) = \OmaxNu(X_0) $ and $ \OMaxNu_{hor} (X_{m_0 +1}) = \emptyset $, we go to the vertical case below.    
	
	Suppose $ \OmaxNu( X_{m_0 +1} ) = \OmaxNu(X_0 ) $ and $ \OMaxNu_{hor} (X_{m_0 +1}) \neq \emptyset $.
	In order to determine the next center we need to study the locus
	$ Y_{m_0 + 1} := \OMaxNu_{hor} (X_{m_0 +1}) $.
	We have a decomposition
	\[
	Y_{m_0 + 1} = Y_{m_0 + 1}^{(0)} \cup Y_{m_0 + 1}^{(1)} \cup \cdots \cup Y_{m_0 + 1}^{(m_0 + 1)} ,
	\ \ \ \mbox{ where:}
	\]
	\begin{itemize}
		\item[{(i)}]	For $ 1 \leq i \leq m_0 $, $ Y_{m_0 + 1}^{(i)} $ denotes 
		the union of irreducible components in $ Y_{m_0 +1} $ that first 
		occurred after the $ i $-th blowing up of the algorithm 
		("label $ i $ components").
		\item[{(ii)}]	$ Y_{m_0 + 1}^{(0)} $ are those irreducible components of 
		$ Y_{m_0 + 1} $ dominating the center $ D_{m_0} $ of the last blowing up. These inherit the label of $ D_{m_0} $.
		\item[{(iii)}]	$ Y_{m_0 + 1}^{(m_0 + 1)} $ are those irreducible components 
		of $ Y_{m_0 + 1} $ lying above the center $ D_{m_0} $ but not dominating. Hence we assign the label $ m_0 + 1 $ to them.
	\end{itemize}
	We also say, $ Y_{m_0 + 1}^{(j)} $ are the 
	{\em irreducible components of $ Y_{m_0 + 1} $ with label $ j $},
	$ 0 \leq j \leq m_0 + 1 $.
	
	Let 
	\[ 
	k := \min \{ j \in \{ 0, \ldots, m_0 + 1\} \mid Y_{m_0 + 1}^{(j)} \neq \emptyset \} .
	\] 
	In the next step, the algorithm repeats the previous, with  
	$ Y_{m_0 + 1}^{(k)} $ now in the place of $ Y_0 $. 
	{Let us point out that in (ii) the components dominating will obtain the label $ k $ instead of $ 0 $, of course.}
	
	\smallskip 
	
	\noindent 
	\underline{\em Vertical case:} \ Suppose $  \OMaxNu_{hor} (X) = \emptyset $.
	We set
	\[
	Y_0 := 
	\OMaxNu (X),
	\]
	and proceed analogous to the horizontal case using $ \OMaxNu$ instead of $ \OMaxNu_{hor} $.
	Note that the vertical case is stable until $ \OmaxNu $ decreases strictly
	(see the remark below).
\end{Construction}

\begin{Rk}
	As we explained before, an irreducible component $ Y' $ in $ \OMaxNu_{hor}(X') $ 
	(resp.~$ \OMaxNu(X') $, if $ \OMaxNu_{hor}(X) = \emptyset$) 
	inherits the label of a component $ D $ contained in $ \OMaxNu_{hor}(X) $ (resp.~$ \OMaxNu(X) $) if $ Y' $ dominates $ D $. 
	In order to obtain this situation, we must blow up an entire irreducible component of $ \OMaxNu_{hor}(X) $ (resp.~$ \OMaxNu(X) $). 
	
	Set $ Y_0 := \OMaxNu_{hor}(X) \subset Z $.
	During the blowing ups in \eqref{eq:inductionblowups}, where
	we prepare $ Y_0 $ to become $ \cB' $-weakly permissible, the centers are strictly
	contained in some irreducible components. 
	Therefore, even if a new irreducible component $ Y' $ in $ \OMaxNu_{hor}(X') $ 
	dominates the center, it never dominates a whole
	irreducible component of $ Y_0 = \OMaxNu_{hor}(X) $.
	Thus we assign to $ Y' $ a new label.
	(The analogous observation applies to $ \OMaxNu(X) $, if $ \OMaxNu_{hor}(X) = \emptyset$).
	
	An irreducible component $ Y' \subset \OMaxNu(X') $ lying above a vertical component of $ D \subset \OMaxNu (X) $ is always vertical.
	Hence, once $ \OMaxNu_{hor} (X) $ is empty, it remains so until $ \OmaxNu(X') < \OmaxNu(X) $ decreases. 
	In other words, in our variant of the {CJS}-algorithm, 
	we first handle the horizontal components 
	and whenever no horizontal component exists, we consider the vertical ones.
\end{Rk}

Let us emphasize that Construction~\ref{CJS-Constr} is a variant of the resolution process by Cossart, Jannsen, Saito,
but the centers are not equal in general,
see Example~\ref{Ex:diff}.
In particular, we have to prove 

\begin{Prop}
	Let $ X $ be a reduced excellent Noetherian scheme of dimension at most two, embedded in a regular scheme $ Z $.
	Suppose that one of the following conditions holds:
	\begin{enumerate}
		\item[(a)] 
		$ X $ has at most dimension one, or
		
		\item[(b)] 
		$ X $ has dimension two and there exists a covering $ Z = \bigcup\limits_{i=1}^m U_i $ such that
		 $ X \cap U_i $ is isomorphic to a hypersurface.  
	\end{enumerate}
	Then the sequence of blowing ups in regular centers constructed in Construction~\ref{CJS-Constr} provides a desingularization of $ X $, i.e., 
	after finitely many blowing ups the strict transform is regular and transversal to the exceptional locus (which is a simple normal crossing divisor). 
\end{Prop}

\begin{proof}
	 Suppose $ X $ has dimension at most one. 
	 Since $ X $ is reduced, its singular locus is at most a finite set of closed points, $ \Sing(X) = \{ p_1, \ldots, p_a \} $.
	 Since blowing ups are isomorphisms outside the center, we may consider each of the closed points in $ \Sing(X) $ separately. 
	 For every $ i \in \{ 1, \ldots, a \}  $, there is a sufficiently small neighborhood $ V_i $ of $ p_i $ such that $ \Sing(X) \cap V_i = \{ p_i \} $. 
	 In particular, the maximal refined order locus of $ X \cap V_i  $ coincides with the maximal Hilbert-Samuel stratum  of $ X \cap V_i  $.  
	After blowing up with center $ p_i $, the singular locus of the strict transform of $ X \cap V_i $ consists of a finite number of closed points and we repeat the previous argument.
	In conclusion, we see that locally our algorithm coincides with the original CJS-algorithm
	and hence the termination follows from \cite{CJS} Theorem~\xch{1.3}.
	
	Let us come to the case (b). 
	Let $ Z = \bigcup\limits_{i=1}^m U_i $ be a covering such that  $ X \cap U_i $ is isomorphic to a hypersurface.  
	If $ X \cap U_i $ is singular, its maximal refined order locus coincides with the corresponding maximal order locus. 
	Furthermore, the maximal Hilbert-Samuel stratum of $ X \cap U_i $ also coincides with the latter.
	Theorem~\xch{3.10} of \cite{CJS} implies that the order and the Hilbert-Samuel function have the same behavior along blowing ups in regular centers contained in the maximal order locus.  
	(Observe that the $ \nu^* $-invariant of \cite{CJS} Theorem~\xch{3.10} is completely determined by the order in the hypersurface case).
	In particular, the refined order locus of the strict transform of $ X \cap U_i $ coincides with the maximal Hilbert-Samuel stratum. 
	
	In \cite{InvDim2} Observation~2.4, the arguments of \cite{CJS} are summarized providing that after a blowing up, one creates at most finitely many closed points or a projective line in the maximal Hilbert-Samuel stratum.
	Since we are in the hypersurface case, this also applies to the locus of maximal refined order.  
	
	By case (a), we know that the preparation process \eqref{eq:inductionblowups} 
	of the lowest label components is finite. 
	By \cite{CJS}~Lemma~\xch{6.30}, the first part of our algorithm is finite:
	More precisely, if we blow up a horizontal component (which must have dimension one),
	then the cited lemma provides that it may happen only finitely many times 
	that there is a projective line dominating the previous center after blowing up.
	(Note that a closed point is necessarily vertical). 
	
	Hence, we may assume that there are no horizontal components in the locus of maximal refined order. 
	Since the latter coincides with the maximal Hilbert-Samuel stratum (as we are in the hypersurface case), 
	the original CJS-centers are the same as those chosen by Construction~\ref{CJS-Constr}.
	Therefore, \cite{CJS} Theorem~\xch{1.3} implies the termination.	
\end{proof}

The following example shows that the proof for the termination of our variant of the CJS-algorithm cannot be reduced to the original result \cite{CJS}, in general. 
More precisely, we construct a scheme with an isolated singularity at a closed point and after blowing up the latter, the locus of maximal order of the strict transform contains a singular curve. 
This does not happen in the original CJS-resolution, where the Hilbert-Samuel function is considered instead of the refined order,
see \cite{InvDim2} Observation~2.4, where the arguments of \cite{CJS} are summarized.

\begin{Ex}
	Let $ a, b, c, d, m \in \ZZ_{\geq 2} $ 
	with
	$ b > 2a $, $ c < d $,
	$ \gcd(a,b) = 1 $,  
	$ m \geq  a(d + 1) $, and  
	$ \gcd(acd,m) = 1 = \gcd(c,d) $.
	Let $ X $ be the scheme given by the ideal 
	\[
	\langle 
	\, 
	x^a - y^b , 
	\,
	(y^{d-c} z^c - u^d)^a  + z^{m} 
	\, 
	\rangle , 
	\]
	where $ (x,y,z) $ are variables and 
	$ u $ is a variable if we are over a field $ k $,
	or $ u = p $ is a prime number if we are working over $ \mathbb{Z} $. 
	The locus of maximal order of $ X $ has ideal $ \langle x, y ,z, u \rangle $.
	If we blow up $ \langle x, y, z, u \rangle $, then we obtain in the $ Y $-chart:
	\[
	x = y_1 x_1,
	\ 
	y = y_1,
	\
	z = y_1 z_1,
	\ 
	u = u_1 y_1
	\]
	and the ideal of the strict transform $ X_1 $ of $ X $ is 
	\[
	\langle 
	\, 
	x_1^a - y_1^{b-a} , 
	\,
	(z_1^c - u_1^d)^a  + y_1^{m-ad} z_1^{m} 
	\, 
	\rangle. 
	\]
	The hypotheses on $ (a,b,c,d,m) $ imply that
	$ b - a > a $ and
	$ m - ad \geq a $.
	Thus, we see that the locus of maximal order of $ X_1 $ contains the {singular} curve given by the ideal
	\[
	\langle x_1, y_1, z_1^c - u_1^d \rangle.
	\]
	On the other hand, one may determine that the maximal Hilbert-Samuel stratum of $ X_1 $ is the closed point corresponding to $ \langle x_1, y_1, z_1, u_1 \rangle $. 
\end{Ex}

\medskip

While the closed point above is also the center of the CJS-resolution, 
it may happen that the center chosen by our variant of the algorithm is different from the original one, see Example~\ref{Ex:diff}.

In conclusion, it is not straight-forward to prove 
the termination of the sequence of blowing ups of Construction~\ref{CJS-Constr} using the techniques of \cite{CJS}. 
Since the present article focuses on aspects of an explicit realizations of the algorithm using {\sc Singular} and GPI-Space,
we leave the proof of the termination in the general case to a later article.

\section{The locus of refined maximal order}
\label{subset:loc_max_ord}

Let $ X \subset Z $ be as before.
In this section we discuss a method to compute the maximal order locus $ \MaxOrd(X) $ of $ X $ from a theoretical viewpoint.
We deduce how to construct the maximal refined order locus $ \MaxNu(X) $, 
from which we obtain its log-variant $ \OMaxNu(X) $ by applying Construction \ref{Constr:vonMaxnachMaxO}. 
Within this, we may restrict ourselves to a finite
affine covering $ \bigcup_{i\in I} U_i $ of $ Z $.
The respective loci are then obtained by gluing together the relevant affine pieces. Note, however, that the maximal value of $ \ord $, $ \nut $, or $ \Onut $ achieved locally on $ X \cap U_i $ is not necessarily a global maximum for $ X $.

\smallskip

We pass to the local situation in a fixed $ U_i $:
Let $ Z = \Spec(A[\bx]) = \mathbb{A}^m_A $, 
where $ A $  is a field or a principal ideal Dedekind domain, and $ (\bx) := (x_1, \ldots, x_{m}) $.
Let $ X = V(\xch{\label{xch:JtoI_X} I_X}) \subset Z $ be the {scheme} defined by a non-zero ideal 
$ \xch{\label{xch:JtoI_X2} I_X} \subset  B :=A[\bx] $.

In Lemma \ref{Lem:ord_x(f)}, we show how to compute the order $ \ord_{I_0} (f) $, 
for $ f \in B $ and $ I_0 := \langle \bx \rangle  \subset B $, using differential operators. 
For a detailed background on (absolute) differential operators, 
we refer to \cite{BernhardThesis} Chapter 2.

Let us recall the definition of {\em Hasse-Schmidt derivatives} 
(cf.~\cite{GiraudPos}, sections 2.5, 2.6, 
or \cite{BernhardThesis} Example (3.3.1)):
Let $ S = k[\bX]  = k [ X_1, \ldots, X_m ] $ be a polynomial ring over a field $ k $.
Let $ F = F(\bX) \in S $.
We introduce new variables $ (\bT) = (T_1, \ldots, T_m ) $ and consider
\begin{equation}
\label{eq:HasseSchmidt_mit_Taylor}
F( \bX + \bT ) = F( X_1 + T_1, \ldots, X_m + T_m) = 
\sum_{\ba \in \ZZ_{\geq 0}^m} F_\ba(\bX)\, \bT^\ba.
\end{equation}
The Hasse-Schmidt derivative of $ F $ by $ \bX^\ba $ is defined by the coefficient of $ \bT^\ba $ in the previous expansion,
\label{xch:atoa1}
\[
\frac{\partial}{\partial \bX^\ba} \,F (\bX) = F_\ba (\bX).
\]
For any $ \ba = (a_1, \ldots, a_m ), \bb = (b_1, \ldots, b_m )\in \ZZ_{\geq 0}^m $ and 
$ \lambda \in k $, we have 
\[
\frac{\partial}{\partial \bX^\ba} \, \lambda \bX^\bb = \lambda \binom{\bb}{\ba} \bX^{\bb-\ba},
\]	 
where $ \binom{\bb}{\ba} = \binom{b_1}{a_1} \cdots \binom{b_m}{a_m} $
and $ \binom{b}{a} = 0 $ if $ b < a $.
In particular, $ \frac{\partial}{\partial \bX^\ba} \,\bX^\ba = 1 $.

Let $ Y $ be one of the variables and $ a, b \in \ZZ_{\geq 0} $.
Since $ \binom{b}{a} \in \ZZ$, we can relate the Hasse-Schmidt derivatives to usual derivatives via the following symbolic computations: \label{Hassesymbolic}
\begin{equation}
\label{eq:HS-der_one}
\frac{\partial}{\partial Y^a} \, Y^b = 
\binom{b}{a} Y^{b-a} = \frac{b!}{a! (b-a)!} \, Y^{b-a}
= \frac{1}{a!} \left( \frac{\partial}{\partial Y} \right)^a \, Y^b.
\end{equation} 

More generally, we have, for 
$ \ba = (a_1, \ldots, a_m ) \in \ZZ_{\geq 0}^m $, that 
\[ 
\frac{\partial}{\partial \bX^\ba} 
= \frac{1}{a_1! \cdots a_m!}
\left( \frac{\partial}{\partial X_1} \right)^{a_1}
\cdots 
\left( \frac{\partial}{\partial X_m} \right)^{a_m}.
\]

\medskip 

In fact, the above construction is valid if we replace the field $ k $ by any domain $ R $.
For $ B =A[\bx] $, 
the usual derivatives by the variables $ x_i $ and the Hasse-Schmidt derivatives generate the $ B $-module
$ \mathrm{Diff}_{A} (B) $ of $ A $-linear differential operators.

Using Hasse-Schmidt derivatives, we can recall a first connection between the order and differential operators:

\begin{Lem}
	\label{Lem:ord_x(f)}
	Let $ R $ be a domain and 
	$ S = R[\bx] $, for $ (\bx) = (x_1, \ldots, x_{m}) $.
	Set $ I_0 = \langle \bx \rangle  \subset S $. 
	We have that, for every $ f \in S \setminus \{ 0 \} $,
	\[
	\label{xch:extra_space} 
	\ord_{I_0} ( f ) = \min \bigg\{ i \in \NN \, \Big|\, \exists \, \ba \in \NN^m : |\ba| = i \, 
	\wedge \,	\frac{\partial}{\partial \bx^\ba} f \notin I_0  \bigg\}.
	\]
\end{Lem} 

{Note that $ R $ is any domain, hence we switch the notation to $ R \subset S $ instead of $ A \subset B $ since the latter is reserved for $ A $ a field or a principal ideal Dedekind domain.}

\begin{proof}
	Recall that $ \ord_{I_0} (f ) =  \sup \{ i \in \NN \mid f \in I_0^i \} $.
	Set $ d := \ord_{I_0} (f ) < \infty $.
	\label{xch:SR_not_AB}Since $ \xch{S} $ is a polynomial ring over $ \xch{R} $, we have that
	$ f = \sum_{\ba} \lambda_\ba \bx^\ba $, for $ \ba \in \NN^m $ and some $ \lambda_\ba \in R $ 
	of which only finitely many are non-zero.
	By definition of the order at $ I_0 = \langle \bx \rangle $, 
	\[ 
	d = \min\{ i \in \NN \mid \exists \, \ba \in \NN^m : |\ba| = i 
	\, \wedge \,	\lambda_\ba \neq 0  \bigg\}.
	\]
	Let $ \ba_0 \in \NN^m $ be such that
	$ | \ba_0 | = d $ and $ \lambda_{\ba_0} \neq 0 $.   
	We have that 
	\[ 
	\frac{\partial}{\partial \bx^{\ba_0}} f = \lambda_{\ba_0} + \sum_{\bb \neq 0 } \lambda_{\ba_0+\bb} \binom{\ba_0+\bb}{\ba_0} \bx^\bb \notin I_0 .
	\]
	By the minimality of $ d $, the assertion follows.
\end{proof}

At a closed point, e.g., given by the maximal ideal 
$ \langle p , x_1 , \ldots, x_m \rangle $, for some prime element $ p \in A $, 
the connection with differential operators is best seen in
the associated graded ring, where $ p $ corresponds to some variable,
say $ P $.

\begin{Ex} \label{ExFibreSing}
	Consider the hypersurface $ X := V(f) \subset \mathbb{A}_\ZZ^2 =: Z $
	given by the polynomial 
	\[ 
	f = 12 - uv^2 \in \ZZ[u,v] .
	\]
	Since $ 12 =  3 \cdot (1 + 3) $ the order of $ f $ at the
	maximal ideal $ \fm := \langle 3 ,u , v \rangle $ is one.
	Therefore $ X $ is not singular at the point corresponding to $ \fm $,
	although $ X $ is singular in the fiber modulo $ 3 $.
	The initial form of $ f $ with respect to $ \fm $ is 
	\[
	in_\fm ( f ) = P \in k_\fm [P, U, V],
	\]
	where $ P := 3 \mod \fm^2 $, $ U := u \mod \fm^2 $, and $ V := v \mod \fm^2 $, and $ k_\fm = \ZZ/3 $.
	As we can see the derivative of $ in_\fm ( f ) $ by $ P $ is
	non-zero.

	On the other hand, $ 12 = 2^2\cdot (1 + 2) $ and at the point 
	corresponding to the maximal ideal 
	$ \fn := \langle 2 ,u , v \rangle $, we have
	$ \ord_\fn ( f ) = 2 $.
	The initial form is
	\[
	F := in_\fn ( f ) = Q^2 \in k_\fn [Q, U, V],
	\]
	where $ Q {\,:=\,} 2 \mod \fn^2 $ ($ U, V $ analogous to above)
	and $ k_\fn = \ZZ/2 $. 
	The derivative by $ Q $ is $ 2 Q = 0 $ in the graded ring. 
	(We also introduced the letter $ Q $ for $ 2 \mod \fn^2 $ 
	in order to avoid confusion between the powers of the
	element $2 \mod \fn^2$ and a factor $ 2 = 0 \in k_\fn $).
	Of course, we have 
	$ \frac{\partial}{\partial Q} ( \frac{\partial}{\partial Q}  F ) = 0 \in \mathfrak{n} $,
	but we have that $  \frac{\partial}{\partial Q^2} Q^2 = 1 \notin \mathfrak{n}$.
	This shows the connection between the order of $ f $ at $ \fn $
	and differential operators.
\end{Ex}

Nonetheless, we need also so called {\em derivatives by constants},
i.e., with respect to elements in the field itself,
in order to compute the locus of maximal order. 
This is related to the notion of $ p $-bases
(see \cite{EGA}, 0$_{\mathrm{IV}}$ (21.1.9), or \cite{BernhardThesis} section 2.2).

\begin{Def}
	Let $ S $ be a ring containing a field of characteristic $ p $
	and $ e \in \NN$, $ e > 0 $.
	A family $ (s_i)_{i\in I} $ in $ S $ is called 
	{\em $ p^e $-independent} (resp.~a $ p^e $-basis) of $ S $ if the 
	family of monomials $ \bs^\ba $ with $ \ba \in \NN^{(I)} $, 
	$ 0 \leq a_i < p^e $ ($ i\in I $),
	is a free family (resp.~a basis) for the $ S^{p^e} $-module $ S $.
\end{Def}

Note that the family $ (s_i)_{i\in I} $ is not necessarily finite.
If the Frobenius $ F: S \to S $, $ b \mapsto b^p $,
is injective (e.g., if $ S $ is a domain) 
and $ (s_i)_{i\in I} $ is a $ p $-basis,
then the monomials $ \bs^\ba $ with $ 0 \leq a_i < p^e $ form a
basis for $ \xch{S} $ considered as $ \xch{S}^{p^e} $-module 
(\cite{BernhardThesis} Remark (2.2.3)). \label{xch:BtoS}

If $ S = k[X_1, \ldots, X_m ] $ is a polynomial ring over a field $ k $
and if $ (\lambda_i)_{i\in I} $ is a $ p^e $-basis for $ k $, 
then the family 
$ (\lambda_i, X_j )_{ i\in I, \, j \in \{1, \ldots, m \} } $ 
is a $ p^e $-basis for $ S $
(\cite{BernhardThesis} Remark (2.2.7)).

Two crucial results are 

\begin{Prop}[\cite{BernhardThesis} Proposition (2.2.5)]
	\label{Prop:Diffoperators}\label{xch:atoa2}
	Let $ S $ be a ring containing a field of characteristic $ p $.
	Let $ ( s_i )_{i\in I} $ be a $ p $-basis of $ S $
	and suppose that the Frobenius $ F : S \to S $ is injective.
	For every multiindex $ \ba \in \NN^{(I)} $, there exists 
	$ \cD_\ba \in \mathrm{Diff}^{\leq |\ba|}_\ZZ (S) $ with
	\[
	\cD_\ba ( \bs^\bb) = \binom{\bb}{\ba} \, \bs^{\bb -  \ba },
	\]
	for all $ \bb \in \NN^{(I)} $.
	Furthermore, every $ \cD \in \mathrm{Diff}_\ZZ^{\leq m}(S) $ is a
	(possibly infinite) sum
	\[
	\cD= \sum_{|\ba| \leq m } c_\ba \cD_\ba 
	\]
	with unique coefficients $ c_\ba \in S $.
\end{Prop}

\begin{Lem}[\cite{GiraudEtude} Lemma 1.2.3, p.~III-4
	(cf.~\cite{BernhardThesis} Lemma (2.3.1)]
	\label{Lem:Diff_op_on_ideal_power}
	Let $ R $ be a ring,
	$ S $ be an $ R $-algebra, and $ \cD\in \mathrm{Diff}^{\leq i }_R (S) $.
	For an ideal $ I \subseteq S $ and $ \ell \geq i $, we have
	\[
	\cD (I^\ell) \subseteq I^{\ell-i}.
	\]
\end{Lem}

\begin{Ex}
	Let $ S = k[X,Y] $, where $ k $ is a non-perfect field of characteristic two.
	Consider $ F = X^2 - \lambda Y^2 $ with $ \lambda \in k \setminus k^2 $.
	The element $ \lambda $ can be extended to a $ 2 $-basis for $ k $.
	We fix a $ 2 $-basis $ (\lambda_i)_{i\in I} $ of $ k $ containing $ \lambda $.
	Then we can speak about the derivative $ \frac{\partial}{\partial \lambda } $.
	
	The maximal order of $ F $ is $ 2 $ and we have
	\[
	\frac{\partial}{\partial \lambda} (X^2 - \lambda Y^2 ) = Y^2.
	\]
	Since all other
	{first order}
	 derivatives are zero, this implies that the locus of maximal order is the closed point $ V ( X, Y ) $.
\end{Ex}

We now come to the computation of the maximal order locus.

\begin{Construction}
	\label{Constr:max_locus}
	Let $ J \subset B = A [\bx ] $ be a non-zero ideal, where $ A $ is a field 
	or a principal ideal Dedekind domain.
	Set $ X := V(J) $.
	Let 
	\[
	d := \max \{ \ord_x (X) \mid x \in X \} \geq 1 .
	\]
	For $ d = 1 $, the locus of maximal order is $ X $.
	Hence let us suppose $ d > 1 $.
	
	If $ B $ contains a field $ k $, $ k \subseteq B $, then, using Proposition \ref{Prop:Diffoperators} and Lemma \ref{Lem:Diff_op_on_ideal_power},
	we get that the maximal order locus for $ X = V(J) $ is given by 
	\[
	\Delta^{d-1} (J) := V ( \cD f \mid f \in J, \cD \in \mathrm{Diff}_\ZZ^{\leq d-1} (B) ).
	\]
	
	Suppose $ B $ does not contain a field 
	(i.e., $ A $ is of mixed characteristic).
	Let $ F := \mathrm{Frac}(A)  $ be the field of fractions of $ A $ and set
	$ B_F := B \otimes_A F =  F[\bx] $.
	If the maximal order of $ V(J \cdot B_F ) $ is equal to $ d $, then
	$ \Delta^{d-1} (J\cdot B_F) $ provides the horizontal irreducible components
	of the locus of maximal order of $ X  = V(J) $.
	
	For the vertical components, let $ p \in A $ be a prime element.
	We pass from $ B = A [\bx ] $ to $ B_p :=
	B \otimes_A \gr_p(A) =  (\gr_p(A))[\bx] $.
	An element $ \sum f_\ba \,\bx^\ba \in B $ is mapped to $ \sum F_\ba \, \bx^\ba \in B_p $,
	where $ F_\ba = F_\ba(P) = \ini_p (f_\ba) \in \gr_p (A) \cong k_p [P] $
	with $ P := \ini_p (p) = p \mod p^2 $ 
	and $ k_p := A/p $ the residue field.
	Note that $ {B_p} \cong k_p[P,\bx ] $ is a polynomial ring over the field $ k_p $. 
	If the maximal order of $ V(J \cdot B_p ) $ is equal to $ d $, then
	$ \Delta^{d-1} (J\cdot B_p) $ provides the vertical irreducible components
	of the locus of maximal order of $ X  = V(J) $ 
	{lying above $ p $}.
	
	Applying this procedure for every prime $ p \in A $, we eventually 
	obtain $ \MaxOrd(X) $. 
\end{Construction}

\begin{Rk}
	\phantomsection\label{Rk:be_careful_diff_op}
	\begin{enumerate}
		\item 
		Since the max order locus has only finitely many irreducible components, 
		there are only finitely many primes that need to be considered. 
		An essential point in implementing the resolution algorithm is to detect
		these "interesting primes".
		In the next section we address this issue for $ A = \ZZ $.
		\item 
		The case $ A = \ZZ$ also explains why we need to pass to $ \gr_p (A) $.
		Let $ p \in \ZZ$ be any prime number, e.g., $ p = 2 $.
		There cannot exist a derivation by $ 2 $ since every derivation is $ \ZZ$-linear.
		Assume there exists $ \cD \in \mathrm{Diff}_\ZZ(B) $ being the derivation by $ 2 $.
		Then $ \cD$ certainly has to fulfill $ \cD(2) = 1 $.
		But the $ \ZZ$-linearity implies
		\[
		1 = \cD(2) = 2 \cD(1) = 2 \cdot 0 = 0 .
		\ \ \  \ \ \ - \ \ \ 
		\mbox{Contradiction.}
		\]	 
		
		\item 
		\xch{In fact, the method in the previous construction can be applied to compute the locus of all points of order at least $ d' $,
		for any $ d' \in \ZZ_+ $ instead of $ d $.}
	\end{enumerate}
\end{Rk}

Using Construction \ref{Constr:max_locus}, we can describe how to determine
the locus of maximal refined order $ \MaxNu(X) $ of $ X $.
As before, we reduce to the affine case and then glue the components.

\begin{Construction}
	\label{Constr:refined_max_locus}
	Let $ J \subset B = A [\bx ] $ be a non-zero ideal, where $ A $ is a field 
	or a principal ideal Dedekind domain.
	Set $ X = V(J) $ {and $ Z = \Spec(B) $}.
	Suppose 
	\[ 
	\maxNu(X) = ( \alpha , \delta  )  \in \NN^2.
	\]
	We provide an inductive construction of $ \MaxNu(X)$ depending on 
	 \[
	a := a(Z) := N - \alpha = \dim(Z) - \alpha  \geq 0
	\ \ \ \ \ \ \ 
	\mbox{(recall Definition \ref{Def:nut})}.
	\]
	If $ a = 0 $, then $ \MaxNu(X) $ coincides with the locus of maximal order $ \MaxOrd(X) $ and 
	we may apply Construction \ref{Constr:max_locus}.
	
	Suppose $ a (Z) \geq 1 $.
	Then, $ J $ is of order one at every point of $ X $.
	In this case, it is possible to descend in the dimension of the ambient space locally at every point.
	Hence, for every point $ q $ in $ \Spec(B) $, there exists a differential operator $ \partial  = \partial(q) $ such that $ (\partial J )\cdot B_{I_q} = B_{I_q} $, where $ \partial J = \langle \partial h \mid h \in J   \rangle$
	and $ B_{I_q} $ denotes the localization of $ B $ at the ideal of $ q $.
	{(Note that, if necessary, we pass from $ B $ to $ B_p $ as in the previous construction).}
	
	Let $ ( f_1, \ldots, f_r ) $ be a {standard basis of} $ J $.
	The latter implies that 
	{there is} a (refined) {finite} open covering of $ X $, 
	{$ X \subset \bigcup_\ell U_\ell $, such that in each $ U_\ell $, we have that} 
	\[
	X_\ell := X \cap U_{\ell} \subset V(f_{j(\ell)}) \cap U_{\ell} =: Z_{\ell},
	\]
	{and $ V(f_{j(\ell)}) $ is regular in $ U_\ell $, for some $ j(\ell) \in \{1, \ldots, r \} $.
		\xch{Let us discuss how to detect such a hypersurface of order one.}\label{xch:indicateonjective} 
		If $ A $ is a perfect field, we can apply the Jacobian criterion for each generator of $ J $ separately and consider the derivatives with respect to the variables $ ( \bx ) $.
		If $ A $ is a non-perfect field, we apply Zariski's Jacobian criterion (\cite{Zariski} Theorem 11),
		which involves a $ p $-basis for $ A $ and uses that there are only finitely many non-zero coefficients appearing in each generator $  f_1, \ldots, f_r $.
		If $ A $ is a principal ideal Dedekind domain that is not a field, there are only finitely many primes $ p_1, \ldots, p_\rho \in A $ appearing in the coefficients of $ f_i $ for each $ i $ and
		for every $ j \in \{ 1, \ldots, \rho \} $, we proceed as follows:
		analogous to Construction~\ref{Constr:max_locus},
		we pass from $ B $ to $ B_{p_j} :=
		B \otimes_A \gr_{p_j}(A) \cong k_{p_j}[P,\bx] $,
		where $ k_{p_j} := A/p_j $ is the residue field at $ p_j $ and $ P $ is a new variable. 
		Then, we apply the Jacobian criterion to detect the desired hypersurface, 
		now also considering the derivative by $ P $.}

	{By Proposition~\ref{Prop:A-D_for_nut}, we have $ \maxNu(X_\ell) \leq \maxNu(X) $.
	If the inequality is strict, we may discard $ {U_\ell} $.
	Otherwise, we have $ \maxNu(X_\ell) = (\alpha, \delta) $ and}	
	since $ \dim(Z_\ell) = \dim(Z) - 1 $,
	we {get} $ a(Z_\ell) = a(Z) - 1 $ 
	(since $ \nut $ and hence $ \maxNu(X) $ does not depend on the embedding).
	By induction, we know how to determine $ \MaxNu (X_{\ell} ) $.
	By gluing together the relevant affine pieces, we obtain then $ \MaxNu(X) $.
\end{Construction}	

Finally, if $ \OmaxNu(X) = ( \alpha , \delta, \sigma  ) \in \NN^3 $, we can apply Construction \ref{Constr:vonMaxnachMaxO} to obtain from $ \MaxNu(X) $ the locus of maximal log-refined order $ \OMaxNu(X) $.

\begin{Obs}
	For practical reason, we like to point out some facts about computing the maximal locus after blowing up:
	Let $ X \subset Z $ be as before.
	Suppose $ \Max(X) \subset X $ is a locus that we may compute from $ X $
	(e.g., $ \Max(X) = \OMaxNu (X) $, or $ \Max(X) $ being the (maximal)
	Hilbert-Samuel locus of $ X $, cf.~\cite{CJS}, Definition \xch{2.35}). 
	
	Let $ \pi : Z' = \Bl_D(Z) \to Z $ be the blowing up in some regular center {$ D \subsetneq \Max(X) $} and denote by $ X' $ the strict transform of $ X $ under $ \pi $.
	Set $ E:= \pi^{-1} (D) $, the exceptional divisor of the blowing up.
	The locus $ \Max(X') $ consists of two types of irreducible components:
	
	\smallskip 
	
	(1) 
	The strict transform $ Y ' $ of $ Y := \Max(X) $ is contained in $ \Max(X') $.

	(2) 
	\parbox[t]{12.5cm}{
	We may have created new components.
	But we know that each of them has to be contained in the exceptional divisor $ E $.}	
	
	\smallskip 
	
	\noindent 
	Since we computed $ Y $ before, we can simply determine its strict transform under $ \pi $ 
	in order to obtain the components of $ \Max(X') $ of type (1).
	Therefore it is only left to determine the new components, i.e., those of type (2).
	The fact that these are contained in $ E $ then helps to simplify the local computations. 
	{The latter is also useful when $ D = \Max(X) $ is the entire locus,
		in which case there is no component of type (1) in $ \Max(X') $ and one has to check whether $ X' $ strictly improved compared to $ X $.}   
\end{Obs}

\section{Algorithmic formulation and examples}
\label{algs}

In the previous sections the theoretical side of the 
desingularization algorithm by Cossart-Jannsen-Saito has been outlined
to give us a solid basis for implementation aspects. We have seen that the 
main loop is a sequence of blowing ups of the arithmetic surface at suitably 
chosen centers. 

From the computational point of view, it is standard to 
represent the given schemes by means of a covering with affine charts, i.e.,
as ideals in a polynomial ring. In each of these charts, we then encounter
a situation $I_Z \subset I_X \subset I_Y \subset \xch{A}[\underline{x}]$ at the beginning and
with the boundary $ \cB$ on $Z$ represented by 
$I_Z \subset I_{\cB} \subset \xch{A}[\underline{x}]$ in later steps of the 
construction. 
\label{xch:R_to_A}

In general, the underlying base ring $\xch{A}$ can be any \xch{principal ideal} Dedekind 
domain or field, as long as the arithmetic operations are practically 
computable; for the purpose of this section, however, we shall 
concentrate on the case of $\xch{A}={\mathbb Z}$.
As the choice of centers is controlled by 
the locus of maximal value of $\Onut$, intersections of charts and patching 
are not an issue in this process. Moreover, blowing up at a given regular 
center is a well-understood standard technique, see e.g. \cite{Sin} or 
\cite{FK1}. 
We can hence focus completely on the practical aspects of the choice of 
centers in this section.

\begin{Rk} 
	As we have seen before, the key ingredient for the choice of the center is 
	the computation 
	of $\OMaxNu(X)$ which was discussed from the theoretical point of view in the 
	last section. 
	More precisely, the center is constructed as follows:
	
	\smallskip 
	
	\noindent 
	{\bf Step 1:} 
	If the maximal order of $X$ is one, the {scheme} can locally be 
	embedded into a regular hypersurface and it thus suffices to find an open 
	covering such that the use of a single such hypersurface on an open set is 
	possible. 
	The equation of this hypersurface can then be added to the 
	generators of the ideal of the ambient space $Z$. 
	Iterating this process, we descend in ambient space as long as 
	we have not reached to minimal appearing
	$a_x$ at any point $x \in X$. 
	At these points the first entry of the invariant
	$N-a_x$ attains a maximal value.
	(Cf.~Constructions~\ref{Constr:max_locus} and \ref{Constr:refined_max_locus}).
	
	\smallskip

	\noindent 
	{\bf Step 2:}
	Again computing the locus of maximal order of $X$ with the new
	ambient space $Z$ resulting from step 1, we reach the locus of maximal refined
	order, i.e., the points where $ \nut(x)=(N-a_x,d_x)$ has maximal value.
	
	\smallskip

	\noindent 
	{\bf Step 3:}
	By taking into account the exceptional divisors, we 
	obtain the locus of maximal log-refined order
	(cf.~Construction \ref{Constr:vonMaxnachMaxO}).

	\smallskip

	\noindent 
	{\bf Step 4:}
	We label the irreducible components of $ \OMaxNu(X) $ using the history of the resolution process
	and detect the locus of components which have smallest label
	(cf.~Construction \ref{CJS-Constr} and the remark after it).
	Since $ X $ is reduced and of dimension two, the latter components have at most dimension one.
	Hence, we either blow up the entire smallest label locus, or we have to prepare it to become \good by blowing up closed points.
	
	\smallskip 
	
	Step 1 and {Step 3} are standard techniques that 
	already appear over fields of characteristic zero 
	(cf.~\cite{BEV}).
	Moreover, they have been realized in implementations 
	(see \cite{BSch}, \cite{ResLIB}) and it requires at most minor modifications to adapt the algorithms to our situation.
	No special features of positive or mixed characteristic appear in these steps.
	The same is true for the labeling process in Step 4 which is of combinatorial
	nature and only requires some diligent book keeping. 
\end{Rk}

Therefore, we completely focus on the algorithm to determine the locus of maximal order in the arithmetic setting, see Algorithm \ref{MaxOrdAlgArith}.
The approach presented in the previous section is purely theoretical as
it requires local computations at each point. 
The analogous problem already occurs in the by far simpler case of
desingularization over a field of characteristic zero 
and we mimic the approach taken there. 
To this end, we first state the algorithm in characteristic zero 
before developing a similar approach in the arithmetic case, see
\underline{Algorithm \ref{MaxOrdChar0}: 
	MaxOrd (char $K = 0$)
	(p.\pageref{MaxOrdChar0})}.
(In there, we denote by $ D(h) $ the principal Zariski open set determined by $ h \in  \QQ[x_1, \ldots, x_n] $).

\begin{algorithm}[h]    
	\caption{MaxOrd (char $K = 0$)}
	\label{MaxOrdChar0}
	
	\begin{algorithmic}[1]
		
		\REQUIRE  $I_Z \subset {\mathbb Q}[x_1,\dots,x_n]={\mathbb Q}[\underline{x}]$ such that $ Z:=V(I_Z) $ is equidimensional and regular,\\
		$I_X \subsetneq {\mathbb Q}[\underline{x}]$ such that $I_Z \subset I_X$, i.e., $ \emptyset \neq X:=V(I_X) \subset Z $
		
		\ENSURE  $ (m,I)$, where $m = \maxOrd(X)$ and
		$I$ ideal describing $\MaxOrd(X)$
		
		\STATE $I_{temp}=I_X$
		\IF {$I_Z == \langle 0 \rangle$}
		\STATE \xch{$I_{max} = \langle 1 \rangle $, $maxord = 0$} \label{xch:Algo:1}
		\WHILE {$I_{temp} ${\,!=\,}$ \langle 1 \rangle$} 
		\STATE $I_{max} = I_{temp}$
		\STATE $I_{temp} = I_{temp} + 
		\langle \frac{\partial f_i}{\partial x_j} \mid
		1 \leq j \leq n\;\;{\rm and}\;\; 1 \leq i \leq r \rangle$,
		where $f_1,\dots, f_r$ generate $I_{temp}$
		\STATE $maxord = maxord +1$
		\ENDWHILE
		\RETURN ($maxord$, $I_{max}$)
		
		\ENDIF	
		\STATE $L$ = \{$\operatorname{codim}(Z)$ square submatrices of Jacobian 
		matrix of $I_Z$\}
		\STATE choose a subset $L_1 \subset L$ such that 
		$X \subset \bigcup_{M \in L_1} D(\operatorname{det}(M))$
		\STATE $maxord=1$, $thisord=0$, 
		$I_{max}= \langle 0 \rangle$
		\FOR {$M \in L_1$}
		
		\STATE denote by $y_1,\dots,y_s$ the system of parameters 
		on $Z \cap D(\operatorname{det}(M))$ induced by variables not 
		corresponding to columns of $M$ 
		\STATE $I_{temp}=I_X$ 
		\WHILE {$ I_{temp} + I_Z  ${\,!=\,}$ \langle 1 \rangle$} 
		\STATE $I_{old}=I_{temp}$
		\STATE $I_{temp}= I_{temp} + 
		\langle \frac{\partial f_i}{\partial y_j} \mid
		1 \leq j \leq s\;\;{\rm and}\;\; 1 \leq i \leq r\rangle$,\\
		where $f_1,\dots, f_r \in {\mathbb{Q}}[\underline{x}] $ correspond to generators of $I_{temp} \cdot  {\mathbb Q}[\underline{x}]/I_Z $
		\\         
		(precise formulation of $ \frac{\partial f_i}{\partial y_j}$
		in Remark \ref{NachDiff})
		\STATE $I_{temp} = $ sat($I_{temp},\operatorname{det}(M)$)
		\STATE $thisord=thisord+1$
		\ENDWHILE
		\IF {$thisord >= maxord$}
		\IF {$thisord == maxord$}
		\STATE $I_{max} = I_{max} \cap I_{old}$
		\ELSE    
		\STATE $maxord=thisord$, $I_{max}=I_{old}$
		\ENDIF
		\ENDIF
		\STATE $thisord=0$
		
		\ENDFOR     
		\RETURN ($maxord$, $I_{max}$)    
	\end{algorithmic}
\end{algorithm}

\smallskip 

Note that we only require that $ Z $ is equidimensional and not necessarily irreducible. 
Since $ Z $ is regular, all irreducible components are disjoint.
From a theoretical viewpoint, it suffices to solve the resolution problem on each component separately and
hence one may reduce to the case $ Z $ irreducible.
In praxis, we may have to deal with $ Z $ having several irreducible components if we pass to a smaller dimensional ambient space as in Construction \ref{Constr:refined_max_locus}.

\begin{Rk}[\cite{BFK} Remark 3.3] \label{NachDiff} 
	As the minor $\operatorname{det}(M)$ is only invertible on
	\xch{the principal open set}\label{xch:prin_open_E_codim} 
	$D(\operatorname{det}(M))$
	(from step 12 of Algorithm \ref{MaxOrdChar0} onwards),
	the differentiation in step 16 is more subtle than it seems at first glance. 
	We start by determining a square matrix $A$ satisfying
	\[
	A \cdot M = q \cdot E_{\operatorname{codim}(Z)},
	\]
	where $q := \operatorname{det}(M)$ \xch{and $ E_{\operatorname{codim}(Z)} $ is the unit matrix of size $ \operatorname{codim}(Z) $}.
	On $D(q)$, $\frac{1}{q} \cdot A$ is precisely the inverse matrix of $M$. As
	system of parameters $y_1,\dots,y_s$, we use the one induced by 
	$\{x_i \mid i\;\; {\rm not \;\; a \;\; column \;\; in} \;\; M\}$. 
	Let $g_1,\dots,g_t \in {\mathbb{Q}}[\underline{x}]$ be a set of generators for $ I_Z $.
	Further, we choose a set of 
	generators 
	$\overline{f_1},\dots,\overline{f_r} \in {\mathbb Q}[\underline{x}]/I_Z$ for
	the ideal $I_X \cdot {\mathbb Q}[\underline{x}]/I_Z$ and choose 
	representatives $f_1,\dots,f_r \in {\mathbb Q}[\underline{x}]$ for these.
	For simplicity of presentation, we assume in the next formula that $M$ 
	involves precisely the last columns of the Jacobian matrix so that the 
	indices of $y_i$ and the corresponding $x_i$ coincide. Then the chain rule 
	provides the following derivatives:
	\begin{equation}
	\label{eq:NachDiff}
	q \cdot \frac{\partial f_i}{\partial y_j} =
	q \cdot \frac{\partial f_i}{\partial x_j} -
	\sum\limits_{\mbox{\footnotesize $ k $ column of $ M $} \atop 
		\mbox{\footnotesize $ \ell $ row of $M$}} 
	\frac{\partial g_\ell}{\partial x_j} A_{\ell k} 
	\frac{\partial f_i}{\partial x_k} \mod I_Z.
	\end{equation} 	
	To discard the extra factor $q$ or more geometrically all components inside 
	$V(q)$, we then need to saturate the resulting ideal with $\langle q \rangle$. 
\end{Rk}

{In the arithmetic setting, 
we first} 
need to determine the primes above which we might find components of $\MaxOrd(X)$ (resp.~$\OMaxNu(X)$). 
{The following example\xch{s} show that %
it is not feasible}
to read off the relevant primes from a given set of generators:

\begin{Ex}
	\label{Ex:Toy_1} \label{xch:Toy_new}
	\xch{Let $f = 3x-y+7z $ and $  g =x-4y+6z $.}
	\begin{enumerate}
	\item 
	Consider 
	$I_Z = \langle g \rangle \subset 
	\langle g,f \rangle = I_X \subset {\mathbb Z}[x,y,z]$.
	{The coefficients of the generators suggest}
	that at most the primes $2$, $3$, $7$ are relevant. However, passing to a 
	different second generator of $I_X$ 
	\xch{we get}  
	$f - 3g = (3x-y+7z)-3\cdot(x-4y+6z)=11\cdot(y-z)$
	\xch{and the locus of order two is $ V (11, y - z , g ) = V (11, y -z, x + 2y) $.}

	\item 
	\xch{Let $ I_Z = \langle 0 \rangle $ and $ I_X  = \langle h \rangle \subset \ZZ[x,y,z] $ be the principal ideal generated by
	\[
		h = f^2 - g^2 
		= 8 x^2 + 2 xy -15 y^2 + 30xz + 34 yz + 13 z^2.
	\]
	If we introduce the new variable  $ \widetilde x := x - 4y + 6 z = g $, then $ f = 3 \widetilde x + 11 (y-z) $.
	Using $ \widetilde y := y - z $, we see that
	\[
		h = (f-g)(f+g) = (2 \widetilde x + 11 \widetilde y) (4 \widetilde x + 11 \widetilde y).
	\]
	The locus of order two is given by the ideal $ \langle 2 \widetilde x + 11 \widetilde y, 4 \widetilde x + 11 \widetilde y \rangle =  \langle 2 \widetilde x, 11 \widetilde y \rangle $.
	In particular, there is the irreducible components $ V (\widetilde x, 11) $, and $ 11 $ does not appear as prime factor in the coefficients of $ h $ as polynomial in $ (x,y) $. }
	\end{enumerate}
\end{Ex}

To systematically determine the primes relevant for the locus of maximal 
order, we now try to mimic the previous algorithm which had been formulated 
in the case of a field of characteristic zero. Over ${\mathbb Z}$, we cannot
expect to obtain the correct locus from this construction -- not even using
Hasse-Schmidt derivatives -- as we are still missing the `derivative' 
with respect to the prime. 
On the other hand, the algorithm for computing $ \MaxOrd(X) $ is blind to vertical
components of it (cf.~Definition \ref{Def:horizon_vertical}) 
{if} 
there are 
horizontal components. 
The latter behavior is not a flaw, 
but in tune with the fact that we have given horizontal components precedence 
over vertical ones in the description of the resolution algorithms in the
previous section.

In an ad-hoc notation, we shall call the primes arising from the algorithm 
{\it interesting} as the bad primes will eventually appear  among those in the
course of the resolution, but not all of the arising primes need to be bad.
  
\color{\farbe} 
  
In \underline{Algorithm \ref{Interesting}: 
	InterestingPrimes
	(p.\pageref{Interesting})}, 
we present an algorithm that detects the set of interesting primes.
\label{xch:Rework_interesting}

\bigskip

\begin{algorithm}[h]
	\caption{InterestingPrimes}
	\label{Interesting}
	\begin{algorithmic}[1]
		\REQUIRE  $I_Z = \xch{\langle g_1,\dots,g_t \rangle}  \subset \ZZ[\underline{x}]$ such that $ Z:=V(I_Z) $ is equidimensional and 
		regular,\\
		$I_X =$ \xch{$ \langle g_1,\dots,g_t,f_1, \ldots, f_r \rangle$} $\subsetneq \ZZ[\underline{x}] $

		\ENSURE   $ \xch{\fP} =\{p_1,\dots,p_\rho\}$ interesting primes for $I_X \supset I_Z$

		\STATE 	resultlist = $\emptyset$

		\color{\farbe} 
				
		\STATE \xch{$I_{temp} = \langle f_1,\dots, f_r \rangle$,} 
		$ I_{int} = \langle 0 \rangle $

		\IF {$I_Z \cap \ZZ${\,!=\,}$\langle 0 \rangle$}
		\STATE resultlist = {\tt primefactors}(generator of principal ideal $I_Z \cap \ZZ$)   (as set) 
		\RETURN resultlist
		\ENDIF

		\color{black} 
		
		\IF {$I_Z == \langle 0 \rangle$}
		\WHILE {$I_{int} == \langle 0 \rangle $}
		\STATE $I_{temp} = I_{temp} + 
		\langle \frac{\partial \xch{F_i}}{\partial x_j} \mid
		1 \leq j \leq n\;\;{\rm and}\;\; 1 \leq i \leq \xch{q} \, \rangle$,
		where \xch{$ F_1,\dots, F_q $ are the generators of $I_{temp} $
                from the last pass through the loop} 
		\STATE $I_{int}= I_{temp} \cap {\mathbb Z}$
		\ENDWHILE
		\STATE resultlist = {\tt primefactors}(generator of principal ideal $I_{int}$)   (as set) 
		\RETURN resultlist
		\ENDIF

		\color{\farbe} 
		
		\STATE resultlist =  resultlist $\cup$ $ \{ p_1, \ldots, p_\alpha  \} $,
		\\
		where  $ \{ p_1, \ldots, p_\alpha  \}  = $ {\tt primefactors}(coefficients appearing in $ g_1, \ldots, g_t $)

		\STATE $L$ = \{$\operatorname{codim}(Z)$ square submatrices of Jacobian 
		matrix of {$ (g_1, \ldots, g_t) $ w.r.t.~$ (x_1, \ldots, x_n ) $}\}

		\STATE choose a subset $L_1 \subset L$ such that 
		$  X \cap D(p_1 \cdots p_\alpha) \subset \bigcup_{M \in L_1} D(\operatorname{det}(M)) \cap D(p_1 \cdots p_\alpha) $
		\FOR {$M \in L_1$}
		
		\STATE denote by $ (y_1,\dots, y_s) $ the system in $ D(\operatorname{det}(M))  \cap D(p_1 \cdots p_\alpha) $ induced by the subsystem of $ ( \bx) $  of variables not 
		corresponding to columns of $M$ 
		\STATE $ I_{int} = \langle 0 \rangle $, $ I_{temp}= \langle f_1,\dots, f_r \rangle $
		\WHILE {$I_{int} == \langle 0 \rangle $}
		\STATE $I_{temp}= I_{temp} + 
		\langle \frac{\partial F_i}{\partial y_j} \mid
		1 \leq j \leq s \;{\rm and}\; 1 \leq i \leq q \, \rangle$,
		where $ F_1 ,\dots, F_q \in \mathbb{Z}[\underline{x}] $ are 
        the generators of $I_{temp}$ computed in the previous pass 
        through the loop 
		%

		\STATE $I_{int} = (I_{temp}+I_Z) \cap {\mathbb Z}$ 
		\ENDWHILE
		
		\IF {$I_{int}${\,!=\,}$\langle 1 \rangle $}
		
		\STATE resultlist = resultlist $\cup$ {\tt primefactors}(generator of principal ideal $I_{int}$)  (as sets)
		\ENDIF 
		
		\ENDFOR
		\RETURN resultlist
		
	\end{algorithmic}
\end{algorithm}

\smallskip 
Note that we demand in line 21 of Algorithm~\ref{Interesting}
that {$ {I_{int}}  \neq  \langle 1 \rangle$} in order to avoid to speak of the prime factors of $ 1 \in \ZZ $. 

\begin{Rk}
	Obviously, the resulting list of interesting primes will heavily depend on the choice of the system of generators $ ( g) $ of $ I_Z $ and $ (g,f) $ of $ I_X $.
	The set $ \{ p_1, \ldots, p_\alpha \} $ is not uniquely determined by $ I_Z $.
	Since our goal is to determine a {\em finite} list of prime numbers, at which we have to test for components in the maximal order locus, the dependence on the mentioned choices is not a problem.
\end{Rk}

In Algorithm \ref{MaxOrdChar0} 
the criterion for the end of the while loop is $ I_{temp} +I_Z  = \langle 1 \rangle $,
while here it is $ I_{int} = (I_{temp} + I_Z) \cap \ZZ \neq \langle 0 \rangle $.
The former can as well be phrased as 
$(I_{temp} + I_Z) \cap  \ZZ  \neq \langle 0 \rangle$, which immediately reveals 
the structural similarity to the latter.

\smallskip

We introduce the following notation.

\begin{Not_num}
	Let $ Z \subseteq \Spec ( \ZZ[x_1,\dots,x_n] ) $  
	be an equidimensional, 
	regular closed subscheme
	and let $ X \subset Z $ be a non-empty closed subscheme of $ Z $. 
	We call a prime number $ p \in \ZZ$ {\em bad 
	(for $ X $)} if the following two conditions hold:
	\begin{itemize}
		\item $ \MaxOrd(X)_{hor} = \emptyset $ and 
		\item there is an irreducible components $ W \subseteq \MaxOrd(X) $ such that $ W \subset V(p) $. 	
	\end{itemize}
	 If $ p $ is not a bad prime, we say $ p $ is a {\em good prime (for $ X$)}.
\end{Not_num}

Whenever it is clear from the context for which subscheme $ X $ we are determining good and bad primes, 
we skip the reference to $ X $.

\begin{Lem}
	\label{Lem:all_bad_primes_interesting}
	Let $ X = V(I_X) \subset Z = V(I_Z) $, for $ I_Z  = \langle g_1, \ldots, g_t \rangle \subset I_X  = \langle g,f_1, \ldots, f_r \rangle \subsetneq \ZZ[\underline{x}] $.
	The set of interesting primes $ \fP =\{p_1,\dots,p_\rho\}$ resulting 
	from Algorithm \ref{Interesting}
	(with the given choice of $ (g, f ) $) contains all bad primes.
\end{Lem}

\begin{proof}	
	Assume $ I_Z \cap \ZZ = \langle m \rangle \neq \langle 0 \rangle $, for $ m \in \ZZ_{\geq 2} $.
	Then $ Z \subset V(m) $ and every prime number, which is bad, has to be a divisor of $ m $.   
	On the other hand, if $ \MaxOrd(X)_{hor} \neq \emptyset $, there is nothing to prove as the set of bad primes is empty in this case, by definition. 
	
	Let us suppose that $ I_Z \cap \ZZ = \langle 0 \rangle $ and  $ \MaxOrd(X)_{hor} = \emptyset $.
	Set $ d := \maxOrd(X) $ 
	and let $ p $ be a bad prime.
	We first consider the case $ I_Z = \langle 0 \rangle $. 
	Since there are no horizontal components in $ \MaxOrd(X) $,
	the locus of points in $ X \times_{\Spec(\ZZ)} \Spec(\QQ) $, which are at least of order $ d $, is empty.
	The latter is obtained by taking $ d-1 $ times the derivatives with respect to the variables $ ( \bx) $.
	This implies, if we apply the same derivatives in $ \ZZ[\bx] $ to generators of $ I_X $,
	then the intersection with $  \ZZ $ of the resulting ideal is non-empty.
	Let $ m \in \ZZ $ be a generator of this principal ideal. 
	By construction, all prime factors of $ m $ are contained in $ \fP $ (Algorithm~\ref{Interesting}, line \xch{7--9}).
	Since the maximal order is $ d $, we may assume $ m > 1 $. 
	
	We claim that $ p $ divides $ m $.
	Suppose this is not the case. 
	In Construction~\ref{Constr:max_locus}, we discussed how to determine the components of the maximal order locus above $ p $. 
	In particular, we repeatedly apply derivatives with respect to $ ( \bx) $. 
	If $ \gcd(p,m) = 1 $, then 
	$ \Delta^{d-1}(J \cdot B_p) = \emptyset $ (using the notation of Construction~\ref{Constr:max_locus}).
	This is a contradiction to $ p $ being a bad prime for $ X $,
	which ends the case $ I_Z =  \langle 0 \rangle $. 
	
	Let us come to $ I_Z = \langle g_1, \ldots, g_t \rangle \neq \langle 0 \rangle $. 
	Let $ \fP_0 := \{ p_1, \ldots, p_\alpha \} \subseteq \fP $ be the set of prime numbers appearing in the coefficients of $ g_1, \ldots, g_t $.
	If $ p \in \fP_0 $, we are done.
	Hence, suppose that $ p \notin \fP_0 $.
	We pass to $ D(p_1 \cdots p_\alpha ) $ and consider a covering of $ Z \cap D(p_1 \cdots p_\alpha ) $ via the Jacobian matrix of $ ( g) $ with respect to $ (\bx) $. 
	Equation~\eqref{eq:NachDiff} in Remark~\ref{NachDiff} provides that the derivatives with respect to the elements $ ( y_1, \ldots, y_s) $ do only involve derivatives by the variables $ (\bx) $.
	Hence, these derivatives can also be considered over $ \QQ $, as in the case $ I_Z = \langle 0 \rangle $.  
	The analogous arguments as in the situation $ I_Z = \langle 0 \rangle $ apply, we only have to replace $ \frac{\partial}{\partial x_i} $ by $ \frac{\partial}{\partial y_j} $ in our reasoning. 
	This ends the proof.
\end{proof}

{The condition that there is no horizontal component is crucial}:

\begin{Ex}
	Consider $ I_X = \langle x^2 - 5^9 y^3 \rangle \subset \ZZ[x,y] $ and $ I_Z = \langle 0 \rangle $.
	It is not hard to see that $ \MaxOrd(X) = V(x,y) \cup V(x,5) $.
	After the first pass through the loop starting at line 3, we get
	\[
	I_{temp} = \langle  x^2 - 5^9 y^3, \ 2x, \ 3 \cdot 5^9 y^2 \rangle 
	\mbox{ and } 
	I_{int} = \langle 0 \rangle.
	\]
	We stay in the loop and obtain after a second run:
	\[
	I_{temp} = \langle  x^2 - 5^9 y^3, \ \ 2x, \ \ 2, \ \ 3 \cdot 5^9 y^2, \ \ 2 \cdot 3 \cdot 5^9 y \rangle = \langle  x^2, \ 2, \ y^2 \rangle
	\mbox{ and } 
	I_{int} = \langle 2 \rangle \neq \langle 0 \rangle .
	\]
	Hence, $ 2 $ is the only interesting prime.
	In particular, we do not detect $ 5 $.  
	This is not a contradiction since $ V(x,y) $ is a horizontal component.
\end{Ex}

We now pick up the toy example of before 
{(Example~\ref{Ex:Toy_1}\xch{(1)})}
to illustrate the algorithm 
to determine the interesting primes 
in the case $ I_Z \neq \langle 0 \rangle $:

\begin{Ex}
	\label{Ex:toy_prime}
	Let $I_Z = \langle g \rangle \subset 
	\langle g,f \rangle = I_X \subset {\mathbb Z}[x,y,z]$ where
	$f =3x-y+7z$ and $g=x-4y+6z$.
	First, $ 2 , 3, 7 $ appear as prime factors in the coefficients of $ g $, hence they are interesting primes. 
	We can use the generator $g$ of $ I_Z $ to 
	eliminate the variable $x$ as it has the coefficient $1$, which also means
	that the corresponding $1$-minor $q$ of the Jacobian matrix has the value $1$.
	(So, we do not have to pass to the locus, where $ 2 \cdot 3 \cdot 7 $ is invertible.)
	This implies that $ ( y, z) $ is the desired subsystem inducing 
	$ ( y_1, y_2 ) $ of line \xch{16} in Algorithm~\ref{Interesting}.
	Using $q=1$, the 
	derivatives of {$ f $} are:
	\[
	\frac{\partial f}{\partial y_1} = \frac{\partial f}{\partial y} -
	\frac{\partial g}{\partial y} \frac{\partial f}{\partial x}=
	-1 - (-4) \cdot 3 = 11
	\ \ \mbox{ and } \ \ 
	\frac{\partial f}{\partial y_2} = \frac{\partial f}{\partial z} -
	\frac{\partial g}{\partial z} \frac{\partial f}{\partial x}
	=
	7 - 6 \cdot 3 = - 11.
	\]
	Thus, the set  of interesting primes is $ \fP= \{ 2,3 , 7,11 \} $. 
\end{Ex}

It is necessary that we collect the primes coming from a given set of generators of $ I_Z $:  

\begin{Ex}
	Let $ p \in \ZZ $ be a prime number. 
	Consider   
	$ I_Z = \langle  p - x^p \rangle \subset \ZZ[x,y,z]$
	and
	$ I_X = I_Z + \langle z^2 + x^p y \rangle $.
	The derivatives by $ (x,y,z) $ are not sufficient to cover $ V(I_X) $ as
	we are missing those points lying above $ p $.
	On the other hand, it not hard to verify, that the locus of maximal order of $ X $ is $ V (p,x,z) $.
\end{Ex} 

To see, how the algorithm proceeds, let us consider a more complex example:

\begin{Ex}
	\label{Ex:InterestingPrimes}
	Let $I_Z := \langle 0 \rangle \subset I_X := \langle 3^2 x^2 - 5^2 y^2  \rangle \subset \ZZ[x,y] $. 
	Here it is easy to check that the maximal order $2$ is attained precisely
	at $ V ( 5, x) \cup V (3, y) \cup V (x,y) $, where the first two components
	are vertical and the last one is horizontal. After the first pass through the
	while loop starting at line 7, we have
	\[
	 I_{temp} =  \langle 3^2 x^2 - 5^2 y^2, \ \ 2 \cdot 3^2 x , 
	\ \ 2 \cdot  5^2 y   \rangle\ \ {\rm and}\;\; 
	I_{int} = \langle 0 \rangle .
	\]
	After the second pass we obtain
	\[
	 I_{temp} = \langle  3^2 x^2 - 5^2  y^2,\ \ 2 \cdot 3^2 x , 
	\ \ 2 \cdot  5^2 y,  \  \  2 \rangle = 
	\langle x^2 + z^2, \ \ 2 \rangle
	\ \ {\rm and}\;\; 
	I_{int} = \langle 2 \rangle \neq \langle 0 \rangle. 
	\]
	Thus, we leave the loop and 
	return only the interesting prime $2$ which is not bad. This is due to the 
	presence of the horizontal component
	$  V (x,y) $. (An algorithm to properly determine this component
	will be discussed in Algorithm \ref{MaxOrdAlgArith}).

	After blowing up with center $ V ( x, y ) $, we consider the $ X $-chart, i.e., the chart with coordinates $ ( x', y' ) = (x, \frac{y}{x}) $.
	Here, again 
	$I_Z = \langle 0 \rangle$, but $I_{\tilde{X}} = \langle 3^2 - 5^2 y'^2 \rangle$.
	Running the algorithm as before, we obtain 
	\[
	I_{temp}= \langle 3^2 - 5^2 y'^2, 2 \cdot 5^2 y' \rangle,
	\]
	which already contains an integer $z = 2 \cdot 3^2$ after
	the first pass; i.e.,
	$ 
	I_{int} = \langle 2 \cdot 3^2 \rangle \neq \langle 0 \rangle  
	$ 
	and the loop stops. 
	Hence, we have
	collected the interesting primes $2$ and $3$, where of course $2$ is still not
	bad, but $3$ really leads to a vertical component. In the other chart, we then see 
	the bad prime $5$ in complete analogy to the discussed chart.
\end{Ex}

The above phenomenon is one of the reasons why we vary the original Cossart-Jannsen-Saito algorithm and give a preference to horizontal components (cf.~Construction \ref{CJS-Constr}).

\color{black} 

\begin{Rk}
	In principle we can use Hasse-Schmidt derivatives (cf.~the idea of the 
	computation in Remark \ref{RemHasseSchmidt} 
	and the corresponding Algorithm 
	\ref{HasseDer} below) instead of the usual 
	derivatives in the preceding algorithm. The ratio of primes contributing 
	to $\MaxOrd(X) $ among the interesting primes would then be higher, because we 
	would not pick up as many primes originating directly from exponents. 
	However, we would like to postpone the rather technical discussion of the
	computation of Hasse-Schmidt derivatives with respect to a system of parameters in the
	presence of a non-trivial $ Z $ as long as possible and first focus on the 
	central task of determining the locus of maximal order.
\end{Rk}

Given a list of interesting primes and the ideals $I_X$ and $I_Z $, we are now 
ready to determine $\MaxOrd(X)$, 
see \underline{Algorithm \ref{MaxOrdAlgArith}: MaxOrdArith (p.\pageref{MaxOrdAlgArith})}. 
As this requires working locally at an 
interesting prime $p$ and differentiating with respect to it, we need {to} introduce

\begin{Construction} 
	\label{Constr:p-adic_to_polynomial}
	{Let $ p \in \ZZ_+ $ be a prime number and $ c \in \ZZ $.
		We have $ c = c' p^\ell $ with $ c' \in \ZZ $, $ \ell \in \ZZ_+ $, and $ \gcd(c',p) = 1 $. 
		We introduce a new variable $ P $ and replace $ c $ by $ c' P^\ell $.}
		
	{Now, let $ I_Z \subset I_X \subset \ZZ[\bx] $ be ideals. 
		We choose a standard basis $ (g_1, \ldots, g_t, f_1, \ldots, f_r) $ for $ I_X $ such that $ ( g_1, \ldots, g_t ) $ is a standard basis for $ I_Z $.
		For every $ f_i, g_j \in \ZZ[\bx] $, we replace the non-zero coefficients as indicated before and denote by $ F_i \in \ZZ[\bx,P] $ (resp.~$ G_j $) the resulting element.
		We then set $ J_X := \langle G_1, \ldots, G_t, F_1, \ldots, F_r \rangle $ and $ J_Z :=  \langle G_1, \ldots, G_t\rangle \subset \ZZ[\bx, P] $.}
\end{Construction}

\begin{algorithm}[!h]
	\caption{MaxOrdArith}
	\label{MaxOrdAlgArith}
	\begin{algorithmic}[1]
		
		\REQUIRE   $I_Z \subset {\mathbb Z}[x_1,\dots,x_n] = \ZZ[\underline{x}]$ such that $ Z:=V(I_Z) $ is equidimensional and 
		regular,
		
		$I_X \subsetneq \ZZ[\underline{x}] $ such that $I_Z \subset I_X$, i.e., $ \emptyset \neq X:=V(I_X) \subset Z $
		
		\ENSURE  $ (m,L)$, where
		$m = \maxOrd(X)$ and
		$L$ list, where $L[i] = (p_i,I_i)$ such that:
		
		{$\bullet$ either $ p_i = 0 $ for all $ i $ and $\MaxOrd(X)_{hor} = \bigcup_i V(I_i)$, if it 
			is non-empty,}
		
		{$\bullet $ or}
		$\MaxOrd(X) = \bigcup_i V(I_i)$ and 
		$I_i$ has been detected locally at {$  p_i > 0  $}

		\STATE MaxOrd0 = {\tt MaxOrd}($I_Z \otimes {\mathbb Q}, I_X \otimes {\mathbb Q}$)
		\STATE $maxord$ = MaxOrd0$[1]$, $I_{max}$ = {MaxOrd0}$[2] \cap {\mathbb Z}[\underline{x}]$
		\STATE RetList[1] = $(0,I_{max})$
		\STATE PrimeList = {\tt InterestingPrimes}($I_Z, I_X$)
		
		\FOR  {$p \in $ PrimeList}
		\STATE 
		{apply Construction~\ref{Constr:p-adic_to_polynomial} for $ I_Z \subset I_X $ 
		(replace each appearing coefficient $c \in {\mathbb Z}$ by $\frac{c}{p^\ell} P^\ell $, for $ \ell $ maximal),}
		denote by $ J_X $ and $ J_Z $ the resulting ideals in $ \ZZ[\underline{x}, P] $

		\IF {$I_Z == \langle 0 \rangle$}
		\STATE DiffList = {\tt HasseDeriv}($ \langle 0 \rangle ,J_X,(\underline{x},P),0$)
		{\small // i-th entry = up to $i$-th derivatives of $ J_X$}

		\STATE $m={\tt size}($DiffList$)$
		\FOR  {$ i \in \{1, \ldots, m \}$}
		\STATE DiffList$[i]={\tt ideal} \big({\tt substitute}($DiffList$[i],P,p) \big) $
		\ENDFOR 
		\WHILE {DiffList$[m] == \langle 1 \rangle $}
		\STATE $ m=m-1 $
		\ENDWHILE 
		\IF {$m>=maxord$}
		\STATE $I_{max}=$ DiffList$[m]$
		\IF {$m>maxord$}
		\STATE RetList = $\emptyset$, \ $maxord = m $, RetList$[1]=(p,I_{max})$
		\ELSIF {{RetList$[1][1]\, $!=$ \, 0 $}} 
		\STATE RetList$[{\tt size}($RetList$)+1]=(p,I_{max})$
		\ENDIF
		\ENDIF 
		\ELSE
		\STATE $L$ = \{$\operatorname{codim}(Z)$ square submatrices of Jacobian 
		matrix of $J_Z {\subseteq\ZZ[\bx, P]}$\} 
		\STATE choose a subset $L_1 \subset L$ such that 
		$ V(J_X) \subset \bigcup_{M \in L_1} D(\operatorname{det}(M))$   
		\STATE $locord=1$
		\FOR {$M \in L_1$}
		\STATE fix $(\underline{y})$ a system of parameters on 
		$ V(J_Z) \cap D(\operatorname{det}(M))$ as in line 12 Algorithm \ref{MaxOrdChar0}
		
		\STATE  DiffList $=$ {\tt HasseDeriv}($J_Z,J_X,\underline{y},M$)
		
		\STATE $m={\tt size}($DiffList$)$
		\FOR  {$ i \in \{1, \ldots, m \}$}
		\STATE DiffList$[i]={\tt ideal} \big({\tt substitute}($DiffList$[i],P,p) \big) $
		\ENDFOR 
		\WHILE {DiffList$[m] == \langle 1 \rangle $}
		\STATE $ m=m-1 $
		\ENDWHILE 
		\IF {$m>locord$}
		\STATE $I_{max}=\,$DiffList$[m]$, $locord=m$
		\ELSIF {$m==locord$}
		\STATE $I_{max}=I_{max}\cap \,$DiffList$[m]$
		\ENDIF
		\ENDFOR
		\IF {$locord >= maxord$}
		\IF {$locord>maxord$}
		\STATE RetList = $\emptyset$, \ $ maxord = locord $, RetList$[1]=(p,I_{max})$
		\ELSIF {{RetList$[1][1]\, $!=$ \, 0 $}} 
		\STATE RetList$[{\tt size}($RetList$)+1]=(p,I_{max})$
		\ENDIF
		\ENDIF 
		\ENDIF
		\ENDFOR
		\RETURN {($maxord$, RetList)}
		
	\end{algorithmic}
\end{algorithm}

Note
that it is not necessary to compute a full $p$-adic expansion, as locally 
in the stalk above $p$ we see a unit {$ c' = \frac{c}{p^\ell}$} due to the choice of $\ell$.

\begin{Rk} 
	{Note that in the output of Algorithm~\ref{MaxOrdAlgArith}, the ideals $ I_i $ are not necessarily prime ideals. 
	Hence, one might prefer to modify the algorithm slightly at the end in order to obtain that each $ I_i $ corresponds to an irreducible component.
	As this is not an essential part, we leave this to an explicit implementation.} 	
\end{Rk} 

\smallskip  

In Algorithm \ref{MaxOrdAlgArith}, there is still one black box to be 
described further: HasseDeriv. In contrast to derivatives in the 
characteristic zero case, we cannot compute them iteratively. Instead, the
Hasse-Schmidt derivatives of a given polynomial $f(\bx )$ are extracted
from $f(\bx + \underline{t})$ as the coefficients of the respective
$\underline{t}^{\alpha}$, as was already mentioned above (see \eqref{eq:HasseSchmidt_mit_Taylor}).
In view of the need to compute derivatives with respect to a regular system of 
parameters for a non-trivial $Z$, we are forced to reconsider the arguments of
Remark \ref{NachDiff}:

\begin{Rk}
	\label{RemHasseSchmidt}
	In the setting analogous to Remark \ref{NachDiff}, we again consider a minor 
	$q=\operatorname{det}(M)$ of the Jacobian matrix of $Z$ and restrict the 
	considerations to $D(q)$, where $q$ is invertible. As 
	before, we can can compute the usual derivatives w.r.t. $ {(\underline{y})} $.
	To pass to Hasse{-Schmidt} derivatives, however, we have to recall the symbolic identity \eqref{eq:HS-der_one},
	\[
	\frac{\partial}{\partial Y^a} \, Y^b =
	\frac{1}{a!} \left( \frac{\partial}{\partial Y} \right)^a \, Y^b.
	\] 
	This stresses that first order Hasse{-Schmidt} derivatives coincide with usual 
	derivatives and it allows us to transform the higher usual derivatives into 
	Hasse{-Schmidt} derivatives at each step of the computation by applying the appropriate
	correction factor. 
	
	In the algorithm to compute the list of Hasse{-Schmidt} derivatives, we thus carry 
	along the information from which derivative a given polynomial originated and
	reuse this when passing to the next derivative. What we are not allowed to 
	apply in this approach is a change of generators of intermediate ideals. 
	Therefore, we need to carry along both the list of already computed 
	derivatives and the corresponding saturated ideal. 
\end{Rk}

In \underline{Algorithm \ref{HasseDer}: HasseDeriv (p.\pageref{HasseDer})}, 
the previous considerations have been 
reformulated closer to an implementable algorithm.

\begin{algorithm}[h]
	\caption{HasseDeriv}
	\label{HasseDer}
	\begin{algorithmic}[1]
		
		\REQUIRE  $I_Z = \langle g_1,\dots,g_t \rangle \subset 
		\ZZ[\underline{x}]$ such that $ Z:=V(I_Z) $ is equidimensional and 
		regular,\\
		$I_X = I_Z + \langle f_1,\dots,f_r \rangle \subsetneq \ZZ[\underline{x}] $, i.e., $ \emptyset \neq X:=V(I_X) \subset Z $,\\
		{$(\underline{y})$} system of parameters on 
		$Z \cap D(\operatorname{det}(M))$,\\
		$M$  $\operatorname{codim}(Z)$ square submatrix of Jacobian matrix 
		of $I_Z$
		
		\ENSURE list $RetList$ such that 
		$RetList[i]=I_Z + \langle {\mathcal D}_a f\mid f\in I_X, |a| \leq i \rangle$ 
		
		\IF {$I_Z== \langle 0 \rangle$}
		\FOR {$j \in \{1,\dots,r\}$}
		\STATE $F_j(\underline{y},\underline{t})=f_j(x_1+t_1,\dots,x_n+t_n)$
		\ENDFOR
		\STATE $i=1$, tempid = $f_1,\dots,f_r$ 
		%
		%
		%
		%
		%
		%
		%
		\WHILE {(($i==1$) OR (tempid {!=} RetList[i-1]))}
		\STATE RetList[i] = tempid
		\FOR {$\ba \in \{ \bb \in \ZZ_{\geq 0}^n \mid |\bb| ==i\}$}
		\STATE tempid = tempid, $\{$ coefficients of $ t^\ba  $ in $F_1,\dots,F_r $ $ \} $
		\ENDFOR
		\STATE $i=i+1$
		\ENDWHILE
		\RETURN {RetList}
		\ENDIF
		\STATE $I_{temp}=I_X$, $Null=(0,\dots,0)$ ($Null$ has \#{$(\underline{y})$} entries)
		\FOR {$1 \leq i \leq r$}
		\STATE L[$i$] = $(f_i,Null)$
		\ENDFOR
		\STATE old = $0$, cur = $r$
		\WHILE {$I_{temp} \cap {\mathbb Z} == \langle 0 \rangle$} 
		\FOR {old $< i \leq $ cur}
		\FOR {$y_j \in {(\underline{y})}$}
		\STATE {($f_{temp}$, note) = L[$i$]}
		\STATE note{$[j]$} = note{$[j]+1$}
		\STATE $f_{temp} = \dfrac{1}{{\mbox{note}}[j]} \cdot 
		\dfrac{\partial f_{temp}}{\partial y_j}$\\
		{\small{(precise formulation of 
				$\frac{\partial f_{temp}}{\partial y_j}$
				in Remark \ref{NachDiff} -- these are usual derivatives)}}
		\STATE L[{\tt size}(L)$+1$] = 
		($f_{temp}$, note) 
		\STATE $I_{temp} = I_{temp}+ \langle f_{temp} \rangle$
		\ENDFOR
		\ENDFOR
		\STATE $I_{temp} = $ sat($I_{temp},\operatorname{det}(M)$)
		\STATE RetList[{\tt size}(RetList)+1] = $I_{temp}$
		\STATE old = cur, cur = {\tt size}(L)
		\ENDWHILE
		\RETURN {RetList}
		
	\end{algorithmic}
\end{algorithm}

%

Even if we start with a trivial ambient space, i.e., $ I_Z = \langle 0 \rangle $,
it may become non-trivial after blowing up:

\begin{Ex}	
	\label{Ex:hyper}
	Consider the hypersurface {$ X \subset Z = \Spec (  \ZZ[x,y] ) $} 
	given by 
	\[ %
	f = 3^2 5^2  + 5 x y + x^3 y^3 \in  \ZZ[x,y].
	\]
	We assume that the boundary is empty.
	The maximal order achieved for $ f $ is $ 2 $ and
	\[ 
	\OMaxNu(f) := \OMaxNu(X) = V ( 3, x, y ) \cup V( 5, x ) \cup V ( 5, y ) ,
	\]
	where the first component is disjoint from the other two.
	All get label $ 0 $.
	One observes that $ \MaxNu(f) $ is not \good and 
	its singular locus is $ V ( 5, x, y ) $.
	Therefore we blow up with center $ V ( 5, x, y ) $.
	Consider the $ Y $-chart, where we set
	$ q := \frac{5}{y} $ and $ x' := \frac{x}{y} $ and $ y' = y $.
	
	It is important to keep in mind the transform of the ambient space, which is
	non-trivial in this chart:  
	$ Z' = \Spec (  \ZZ[x',y', q] / \langle 5 - q y' \rangle  ) $.
	The exceptional component is given by $ V( y' ) $, 
	hence the boundary is $ \cB' = \{ V ( y' ) \} $.
	The strict transform of $ f $ is 
	\[ 
	f' = 3^2 q^2  + q x' y' + {x'}^3 {y'}^4.
	\]
	Note that $ \OMaxNu(f') $ consists of the strict transform of $ \OMaxNu(f) $ in $ Z ' $  
	and components contained in the exceptional divisor.
	Hence we obtain 
	\[ 
	\OMaxNu(f') =  V( q, x' ) \cup V ( q, y' ) \xch{.} \label{xch:period_4.15}
	\]
	The first component is the strict transform of $  V( 5, x ) $ 
	and has label 0. 
	The second is contained in the exceptional component and 
	thus gets the new label 1. 
	
	Note: In the computations, the component $  V ( 3, x',y' ) $ of the maximal order locus has been 
	handled, but we do not see it in this chart as its intersection with the non-trivial ambient space
	is empty here: $\langle 3,x',y',5 - q y' \rangle = \langle 1 \rangle$ as $\operatorname{gcd}(3,5)=1$.
	
	Since $ V( q, x' ) $ is \good 
	the algorithm chooses this component as the center for the next blowing up.
	We leave the other charts and the remaining resolution process to the reader.
\end{Ex}

Let us point out again that our variant is different from the original 
CJS-resolution.
Since we are using $ \nut $ instead of the Hilbert-Samuel function, 
we cannot control the maximal value of the Hilbert-Samuel function along our variant of the resolution process:

\begin{Ex}
	\label{Ex:diff} 
	Let {$ X \subset Z = \mathbb{A}^3_{\ZZ}  $} 
	be the {scheme} 
	defined by the ideal 
	\[ 
	J = \langle x^2 - y^{17}, p^5 - y^2 z^6 \rangle \subset \ZZ[x,y,z] ,
	\] 
	where $ p \in \ZZ $ is a prime number.
	We assume that the 
	boundary is empty $ \cB = \emptyset $.
	We have $ \maxNu(X) = ( 4,2) $ and
	\[
	\OMaxNu (X) = \MaxNu(X) =  V ( p, x,y ). 
	\]
	In contrast to this, the maximal Hilbert-Samuel locus is 
	$ V (p, x,y,z ) $.	
	
	The algorithm proposes as first center $  V ( p, x,y ) $.
	After blowing up, we consider the $ Y $-chart, 
	where we use the notation 
	$ p' := \frac{p}{y} $, $ x' := \frac{x}{y} $, $ y' = y $, 
	and $ z' := z $.
	The strict transform of the ambient space is 
	$ Z' = \Spec (  \ZZ[p',x',y',z'] / \langle p - p' y' \rangle  ) $
	and $ X' $ is given by the ideal
	\[
	J' = \langle x'^2 - {y'}^{15}, \ {p'}^5 {y'}^3 - {z'}^6 \rangle
	\xch{.} \label{xch:period_4.16}
	\]
	Hence, $ \maxNu(X') =  \maxNu(X) = ( 4,2) $.
	Note that the maximal order of the second polynomial increased from $ 5 $ to $ 6 $ after the blowing up. 
	This implies that the maximal value achieved by the Hilbert-Samuel function increased%
	\footnote{The precise value is $ (1,4,9,16,25,36,48,\dots) $, whereas it was $(1,4,9,16,25,35,45,\dots) $ before blowing up.}, 
	which is not surprising since we have blown up a center
	that is not contained in the maximal Hilbert-Samuel locus.
	
	We have $ \OMaxNu ( J' ) = V (x',y',z') $ which is a \good and hence the next center.
	We leave the other charts and the remaining resolution process to the reader.
\end{Ex}

{To end this section, let us outline for the toy example, why it is crucial to use a standard basis in Construction~\ref{Constr:p-adic_to_polynomial}.}

\begin{Ex}
	{Let $I_Z = \langle g \rangle \subset 
		\langle g,f \rangle = I_X \subset {\mathbb Z}[x,y,z]$ where
		$f =3x-y+7z$ and $g=x-4y+6z$.
		In Example~\ref{Ex:toy_prime}, we have seen that $ p = 11 $ is the only interesting prime.
		If we replace each coefficient $ c $ in $ f $ and $ g $ by $ c' P^\ell $ for $ c' \in \ZZ $ prime to $ 11 $, we obtain the same elements.
		With the same argument as in Example~\ref{Ex:toy_prime}, $ (y,z,P) $ induces a regular system of parameters $ (y_1, y_2, y_3) $ at every point of $ V(I_X) $
		and we have $ \frac{\partial f}{\partial y_1} = P $ and $ \frac{\partial f}{\partial y_2} = -P $.
		Further, since $ P $ does not appear, $ \frac{\partial f}{\partial y_3} = 0 $.
		This suggests that the locus of maximal order is given by the ideal $ \langle g, f, 11 \rangle = \langle g, 11 \rangle $,
		where the last equality holds since $ f = 3g + 11(y-z) $.
		This cannot be true.}
	
	{On the other hand, we may replace $ f $ by $ h := 11 (y-z) $ and $ (g,h) $ is a standard basis for $ I_X $.
	Applying Construction~\ref{Constr:p-adic_to_polynomial} yields the element $ H = P (y- z ) $. 
	Computing the derivatives leads to the correct maximal order locus $ V (g, 11, y-z) $.
	The details are left as an exercise to the reader.}
\end{Ex}

\section{Exploiting the Naturally Parallel Structure}

In this section, we discuss the parallelization of the algorithm for finding
the center of the chosen resolution strategy. The ubiquitous use of coverings 
by charts already suggests a high potential for parallelization; instead of 
starting with an arbitrarily chosen covering by a subset of the minors of the
Jacobian matrix of $I_Z$, we can also start with all possible choices, run 
them in parallel and terminate as soon as all of $X$ has been covered. As we
aim to use GPI-Space \cite{GPI} as the workflow-management system for managing and 
scheduling our parallel approach, we rephrase our algorithms in the language
of Petri nets, which is slightly unusual for algebraic geometry, but has 
already proved useful in \cite{BDF}. 

\subsection{Petri nets and parallel implementation}
Petri nets are models of distributed systems.
In its basic form, a Petri net is a bipartite, directed, finite graph. 
Its vertices are \emph{places}, which are denoted by circles in graphical representation, 
and \emph{transitions}, which are denoted by rectangles. 
The edges of the graph are called \emph{arcs}. 
The places from which there is an arc to a certain transition are called 
the \emph{input places} of that transition. 
Similarly, the places which can be reached with an arc from 
a certain transition are the \emph{output places} of that transition.
There can also be \emph{tokens} on the places, depicted by small solid circles. 
A token can be thought of carrying a piece of information 
(formally: it is of a certain \emph{color}) 
that is being processed by the Petri net. 
A configuration of tokens is called a \emph{marking} of the net.
\emph{Colored Petri nets} are in fact an extension to the original concept where the tokens on one place are indistinguishable. 
There are actually numerous extensions of Petri nets; 
some of them provide important features for describing computations  
which cannot be phrased with the language of basic Petri nets. 
We describe the properties of the Petri nets which are used in GPI-Space.
The latter is a parallel environment that handles the workflow management 
of our implementation.

A transition of a Petri net is called \emph{enabled} 
if there is at least one token at each of its input places. It can \emph{switch}
(or \emph{fire}). When it switches, the transition consumes one token from 
every input place and places new tokens on every output place. Note that the 
Petri net itself makes no statement when an enabled transition will switch 
(if it does at all) and which tokens it will consume if there are several 
to choose from. In this sense, Petri nets are executed in a non-deterministic 
way (the actual choices are up to the underlying implementation).

Let us now consider an extremely simplified example to get used to the 
terminology of Petri nets:

\begin{Ex} \label{compositionofmaps}
	Consider a composition of functions $ X \xrightarrow{f} Y \xrightarrow{g} Z$ between some sets $ X, Y, Z $. 
	This translates into the following Petri net
	(only showing places and transitions):
	\begin{petrinet}
		\node[placeL]      (pX) [label=below:$p_X$]          {};
		\node[transitionL] (f)  [right of=pX]                            {$f$};
		\node[placeL]      (pY) [label=below:$p_Y$,right of=f] {};
		\node[transitionL] (g)  [right of=pY]                            {$g$};
		\node[placeL]      (pZ) [label=below:$p_Z$,right of=g] {};
		
		\path[->]
		(pX) edge (f)
		(f)  edge (pY)
		(pY) edge (g)
		(g)  edge (pZ)
		;
	\end{petrinet}
	So far, we are hiding the tokens.
	If there is a token in $ p_X $, it can be
	consumed by the transition $ f $ which then provides 
	at least one token in $ p_Y $,
	see below for further explanations.
	A snapshot of the Petri net might look as follows:	
	\begin{petrinet}
		\node[placeL]      (pX) [label=below:$p_X$, tokens = 3]          {};
		\node[transitionL] (f)  [right of=pX]                            {$f$};
		\node[placeL]      (pY) [label=below:$p_Y$,right of=f] {};
		\node[transitionL] (g)  [right of=pY]                            {$g$};
		\node[placeL]      (pZ) [label=below:$p_Z$,right of=g , tokens = 1] {};
		
		\path[->]
		(pX) edge (f)
		(f)  edge (pY)
		(pY) edge (g)
		(g)  edge (pZ)
		;
	\end{petrinet}
	That means currently, there are three tokens in $p_X $ 
	waiting to be consumed by $ f $ and one in $ p_Z $. Transition $g$
	cannot fire, as there are no token in its input place $p_Y$, whereas
	transition $f$ is enabled as a consequence of the tokens in $p_X$.
	Upon firing, transition $f$  consumes one token from place $p_X$,
	representing an element $x \in X$, and produces one token on place 
	$p_Y$, representing a token $f(x) \in Y$.
\end{Ex}

Example \ref{compositionofmaps} also shows further properties of Petri nets: 
One important point is that all dependencies are local, at least in theory. 
Every transition only needs to know about its input and output places, 
for both the data (color of tokens) and control dependency (presence of 
tokens). There is no need for any kind of global synchronization or global 
clock signal. This also means that in a real setting (i.e., in GPI-Space 
through which the computer executes the Petri net), there is usually no 
global state during execution that can be described with a marking. 
In the model, each switching is an atomic, instantaneous process,
i.e., it is a step that {cannot} be further divided into smaller parts and 
it is processed immediately without time consumption. 
In a real system, the execution of a transition will take a finite amount 
of time and there will be several transitions executing at a given point in 
time, in general. However, it is always possible to give an equivalent 
\emph{switching sequence} without overlaps.

This leads to the second important point: Petri nets automatically come with 
parallelism. There is both data parallelism, that is, there can be several 
instances of one transition running at the same time (on different CPU 
cores or even on different machines), but also task parallelism, that is, 
instances of different transitions can be running simultaneously.

\smallskip

Keeping this background in mind, we can now state a Petri net corresponding 
to a simplified version of the embedded resolution process:

\begin{Ex} \label{Grobstruktur}
	Roughly speaking a resolution of singularities via blowing ups in regular 
	centers has the following structure:
	Given $ X \subset Z $ as in the previous sections, first determine the 
	center $ C $ of the first blowing up. Then, if $X$ is not already
	resolved (i.e., if $ C \subsetneq X $), we blow up at the center $ C $. 
	This process is then repeated for the transform of $ X \subset Z $ under 
	the blowing up.	Eventually, we reach the point that all singularities are 
	resolved -- at least in those cases, where desingularization by a finite
	number of blowing ups is not an open problem.
	
	As we mentioned before, the resolution data is collected in local affine 
	charts. In other words, given the initial data $ X  \subset Z $, we first 
	have to split it into affine charts, each of which corresponds to a token 
	(in place a below). Furthermore, after blowing up (i.e., at place c below), we 
	also have to split the data into finitely many affine charts (i.e., we 
	produce several tokens from one). Those of the affine charts (i.e., tokens) 
	that are resolved are detected at place b and glued with the already handled
	resolved charts. The resulting token holding the glued object is put into
	place d, ready for gluing the next chart to it. If the token on place d 
	already describes the entire resolution of $X$, the transition Heureka 
	fires and initiates the clean-up and termination of the Petri net. The 
	output is
	an Embedded Resolution of Singularities $ERS(X,Z,\cB)$, where $ \cB$ is some boundary on $ Z $.
	
	These explanations provide the simplified Petri net \xch{of Figure~\ref{Fig:ERS}}, where we use the obvious 
	notation. We mark the beginning (resp.~end) 
	by an upper (resp.~lower) half-disc.   	

	
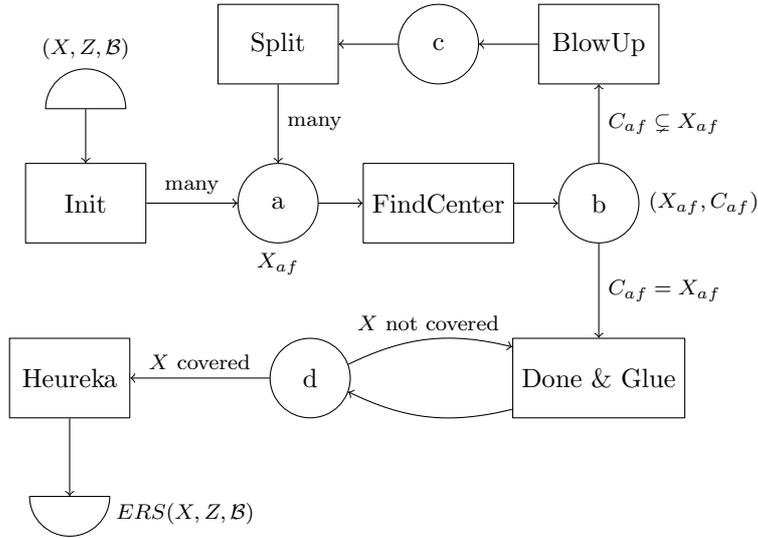
\begin{figure}
	\begin{petrinet}
		\node[startplaceL] (start) [label=above:{\footnotesize $(X, Z, \cB)$}]                 {};
		\node[transitionL] (init)  [below of=start, node distance=0.7\ndL] {Init};
		\node[placeL]      (pla)   [right of=init, node distance=1.2\ndL, label=below:{\footnotesize $X_{af}$}]  {a};
		\node[transitionL] (find)  [right of=pla]   {FindCenter};
		\node[placeL]      (plb)   [right of=find, label=right:{\footnotesize ($X_{af},C_{af}$)}]  {b};
		\node[transitionL] (blow)  [above of=plb]   {BlowUp};
		\node[placeL]      (plc)   [left of=blow]   {c};
		\node[transitionL] (split) [left of=plc]    {Split};
		\node[transitionL] (coll)  [below of=plb] {Done \&  Glue};
		\node[placeL]      (pld)   [left of=coll, node distance=1.8\ndL]  {d};
		\node[transitionL] (heur)  [left of=pld, node distance=1.5\ndL] {Heureka};		
		\node[endplaceL]   (end)   [below of=heur, node distance=0.7\ndL, label=right:{\footnotesize $ERS(X, Z, \cB)$}]   {};

		\path[->]
		(start) edge (init)
		(init)  edge node [above] {\footnotesize many} (pla)
		(pla)   edge (find)
		(find) edge (plb)
		(plb)   edge node [right] {\footnotesize $C_{af} \subsetneq X_{af}$} (blow)
		(blow)  edge (plc)
		(plc)   edge (split)
		(split) edge node [right] {\footnotesize many} (pla) 
		(plb)   edge node [right] {\footnotesize $C_{af} = X_{af}$} (coll)
		(pld)  edge [bend left=20] node [above] {\footnotesize $X$ not covered} (coll)
		(coll) edge [bend left=20] (pld)
		(pld)  edge node [above] {\footnotesize $X$ covered} (heur)
		(heur) edge (end)
		;
	\end{petrinet}
	\caption{\xch{Simplified Petri net for an embedded resolution of singularities.}}
	\label{Fig:ERS}
	\end{figure}
\end{Ex}

Note that we have omitted several details in the above Petri {net}, e.g. the use
of place d can only work properly, if a token representing the empty set is 
placed there at the beginning. To discuss all technicalities of the use of 
Petri nets, such as the use of 
conditions and parameter dependent behavior, is far beyond the scope of the
current article.
For an introduction to parallelization using Petri nets for algebraic geometers
we refer to section 3 and 4 of \cite{BDF}. 
We only state simplified Petri nets in the following sense: only the 
essential steps in the respective 
algorithm are presented, some of the transitions may themselves represent
Petri nets, like {\tt FindCenter} above, whereas others, like {\tt BlowUp},
represent code which is executed in the back-end {\sc Singular} \cite{Sin} and seen as
atomic from the perspective of GPI-Space. Moreover, we take the liberty to
add annotations for the data (i.e., coloring of tokens) at those places, 
where it contributes to a better understanding, and for the arcs, where
the transition is subject to a condition or where more than one output token
is created.

\subsection{Petri nets within the desingularization process}

In Example \ref{Grobstruktur}, we have already seen the inherent parallelity
through the use of charts, but we have not yet discussed the very heart of the
algorithm: the choice of center. The \xch{Petri net in Figure~\ref{Fig:FindCenter}} describes the 
algorithm {\tt FindCenter}, which is itself of sequential nature and directly
corresponds to the steps outlined in section \ref{algs} as can easily be
observed from its Petri net. The input- and output-tokens of this net are
precisely of the structure of the respective tokens at the input and output
place for {\tt FindCenter} in the Petri net in Example \ref{Grobstruktur}.

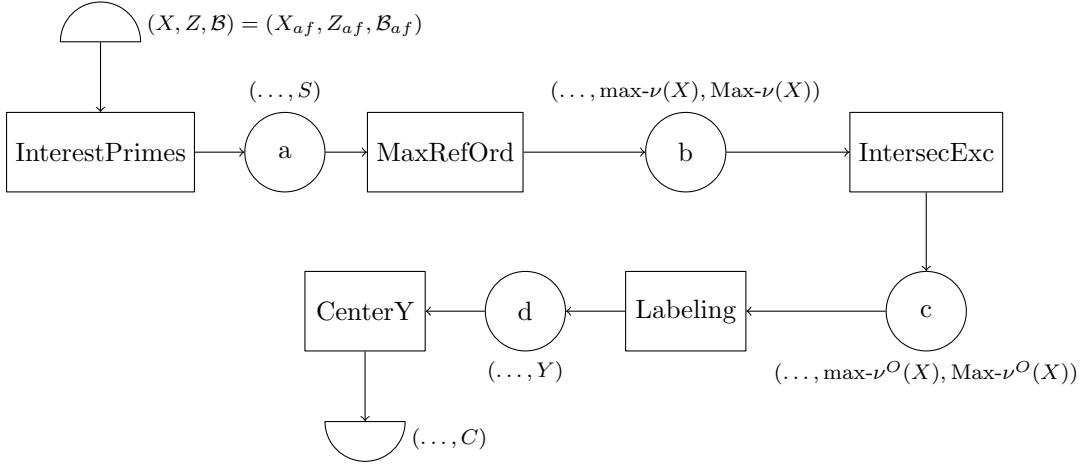
\begin{figure} 
\begin{petrinet}
	\node[startplaceL] (start) [label=right:{\mbox{\footnotesize $(X,Z,\cB) = (X_{af},Z_{af},\cB_{af})$}}]                  {};
	\node[transitionL] (iprime)  [below of=start, node distance=0.65\ndL] {InterestPrimes};
	\node[placeL]      (pla)   [right of=iprime, node distance=1.15\ndL, label=above:{\mbox{\footnotesize $(\ldots,\xch{\fP})$}}]  {a};
	\node[transitionL] (max)  [right of=pla]   {MaxRefOrd};
	\node[placeL]      (plb)   [right of=max, node distance=1.5\ndL, label=above:{\mbox{\footnotesize $(\ldots,\maxNu(X),\MaxNu(X))$}}]  {b};
	\node[transitionL] (inter) [right of=plb, node distance=1.5\ndL]   {IntersecExc};
	\node[placeL]      (plc)   [below of=inter, node distance=0.70\ndL, label=below:{\mbox{\footnotesize $(\ldots,\OmaxNu(X),\OMaxNu(X))$}}]  {c};
	\node[transitionL] (labell)  [left of=plc, node distance=1.5\ndL] {Labeling};
	\node[placeL]      (pld)   [left of=labell, label=below:{\mbox{\footnotesize $(\ldots,Y)$}}]  {d};
	\node[transitionL] (centY)  [left of=pld] {CenterY};
	\node[endplaceL]   (end)   [below of=centY, node distance=0.6\ndL,label=left:{\mbox{\footnotesize $(\ldots, C)$}}]   {};

	\path[->]
	(start) edge (iprime)
	(iprime)  edge (pla)
	(pla)   edge (max)
	(max)  edge (plb)
	(plb)   edge (inter)
	(inter) edge (plc)
	(plc)   edge (labell)
	(labell)  edge (pld)
	(pld) edge (centY)
	(centY) edge (end)
	;  
\end{petrinet}
\caption{\xch{Petri net for {\tt FindCenter.}}}
\label{Fig:FindCenter}
\end{figure}

\noindent 
Let us point out that, in the Petri net, $ Y $ denotes the locus of those irreducible components in $ \OMaxNu(X) $  that are of lowest label.

The transitions {\tt IntersecExc} and {\tt Labeling} are implemented in the
back-end {\sc Singular} and hence seen as atomic in this context. The purpose
of {\tt CenterY} is the choice of a center to resolve $Y$, which is of 
dimension at most $\dim(X) -1$. In a general setting, this is
a recursive call to the resolution. But in the special case of 
$\dim(X)=2$, this only handles at most a finite number of 
singular points of $Y$ and we consider it as simple call to {\sc Singular}.
Note that we will get $ C = Y $ at some step of the algorithm, i.e., that the entire locus $ Y $ is the next center.

This leaves two important building blocks to discuss: {\tt InterestPrimes}
and {\tt MaxRefOrd}, both of which exhibit a naturally parallel structure. 
\xch{In Figure~\ref{Fig:IntPrimes}, we} first discuss the Petri net of {\tt InterestPrimes} which corresponds to 
the algorithm InterestingPrimes of the previous section.

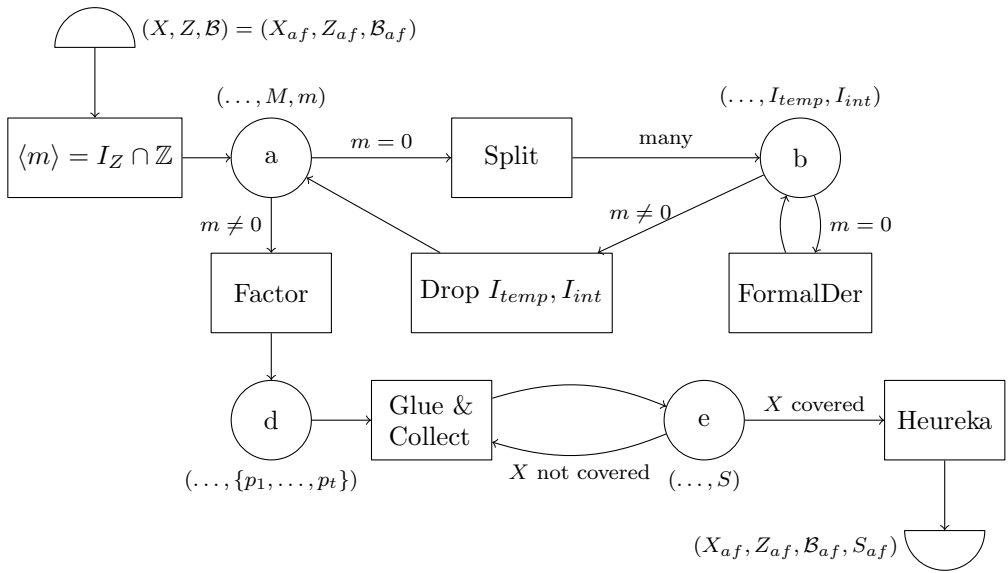
\begin{figure} 
\label{xch:Petri}
\begin{petrinet}	
	\node[startplaceL] (start) [label=right:{\mbox{\footnotesize $(X,Z,\cB) = (X_{af},Z_{af},\cB_{af})$}}]                 {};
	\node[transitionL] (capint)  [below of=start, node distance=0.65\ndL] {$ \langle m \rangle = I_Z\cap \ZZ  $ };
	\node[placeL]      (pla) [right of=capint, node distance=1.1\ndL,label=above:{\mbox{\footnotesize $(\ldots,M, m, \xch{\fP})$}}]  {a};
	\node[transitionL] (fac)  [below of=pla, node distance=0.85\ndL] {Factor};
	\node[placeL]      (pld) [below of=fac, node distance=0.75\ndL,label=below:{}]  {d};
	\node[transitionL] (coll)  [right of=pld, node distance=1\ndL] {\parbox{32pt}{Glue \&\\ Collect}};
	\node[placeL]      (ple)   [right of=coll, node distance=1.7\ndL,label=below:{}]  {e};
	\node[transitionL] (heur)  [right of=ple, node distance=1.5\ndL] {Heureka};
	\node[endplaceL]   (end)   [below of=heur, node distance=0.6\ndL, label=left:{\mbox{\footnotesize $ (X_{af},Z_{af},\cB_{af}, \xch{\fP_{af}})$}}]   {};
	
	\node[transitionL]   (split)   [right of=pla, node distance=1.5\ndL]   {Split};
	\node[placeL]   (plb)   [right of=split, node distance=1.8\ndL, label=above:{\mbox{\footnotesize $(\ldots, I_{temp},I_{int})$}}]   {b};
	\node[transitionL]   (der)   [below of=plb, node distance=0.85\ndL]   {FormalDer};
	\node[transitionL] (drop)  [below of=split, node distance=0.85\ndL] {Drop $ I_{temp},I_{int} $};

	\path[->]
	(start) edge (capint)
	(capint)  edge (pla)
	(pla) edge node [left] {\footnotesize $ m \neq 0 $}  (fac)
	(fac) edge (pld)
	(pld) edge (coll)
	(ple)  edge [bend left=20]  node [below] {\footnotesize $X$ not covered} (coll)
	(coll) edge [bend left=20] (ple)
	(ple) edge  node [above] {\footnotesize $X$ covered} (heur)
	(heur) edge (end)
	(pla) edge  node [above] {\footnotesize $ m = 0 $} (split)
	(split) edge node [above] {\mbox{\footnotesize many}}  (plb)
	(plb) edge [bend left=20] node [right] {\footnotesize $ m = 0 $} (der)
	(der) edge [bend left=20] (plb)
	(plb) edge node [left] {\footnotesize $ m \neq 0 $} (drop)
	(drop) edge (pla)
	;
\end{petrinet}
\caption{\xch{Petri net for {\tt InterestingPrimes.}}}
\label{Fig:IntPrimes}
\end{figure}

Note that for consistency of coloring at place a, the matrix $ M $ is initialized to the zero matrix in the very first transition of the net. 
\xch{Further, we add the empty list $ \fP$, in which we will collect the interesting primes.}
\xch{In {\tt Split}, we fix generators for $ I_Z $ and add the prime factors of the appearing coefficients to the list $ \fP $. Following this, {\tt Split} creates}
the covering by affine charts arising from the minors of the Jacobian matrix
of $I_Z$, contains a call to {\sc Singular}  and fills in a correct value for 
the matrix $M$ for each of the generated tokens. The iteration of purely 
formally forming the derivatives is represented by {\tt FormalDer}, 
which relies on functionality of the back-end {\sc Singular} again. Having 
found an integer $m \neq0$, we need to find the prime factors, combine the
knowledge from all charts and return the result. 

\smallskip 

The next transition in the Petri net of {\tt FindCenter} is the computation 
of the locus of maximal refined order {\tt MaxRefOrd}, which is represented by
the Petri net \xch{in Figure~\ref{Fig:MaxRefOrd}}. In reality, the output place of {\tt InterestPrimes} and
the input place of {\tt MaxRefOrd} are identical; for presentation as 
separate Petri nets, we therefore start {\tt MaxRefOrd} by a transition
which actually only places its input token into its output place a. 

As the first entry of the desingularization invariant represents the number 
of descents upon finding maximal order one, the main part of this Petri net
is the loop which computes the maximal order locus and then passes to the
descent in ambient dimension, if the maximal order was $1$. If the maximal
order exceeded $1$, the transition {\tt Glue\&Collect} fires instead, compares
the maximal order to the current maximal value, identifies the current maximal 
order and places the locus of this order (as computed up to this point) into
the place c. Note that the zero ideal has order $\infty$, which also triggers
{\tt Glue\&Collect}, but the transition itself handles this special case 
appropriately, as it first compares the first entry of the invariant.

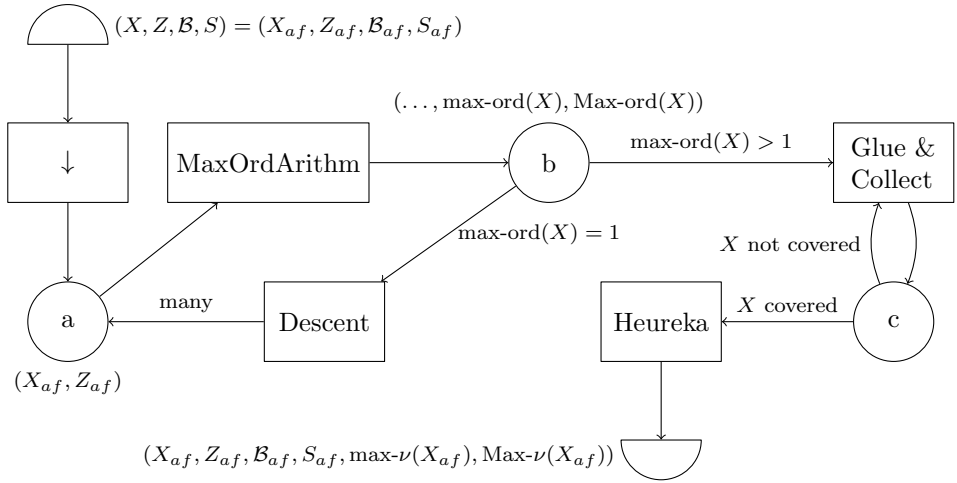
\begin{figure} 
\begin{petrinet}
	\node[startplaceL] (start) [label=right:{\mbox{\footnotesize $(X,Z,\cB, \xch{\fP}) = (X_{af},Z_{af},\cB_{af}, \xch{\fP}_{af})$}}]                  {};
	\node[transitionL] (down)  [below of=start, node distance=0.75\ndL] {$ \downarrow $};
	\node[placeL]      (pla) [below of=down, node distance=1\ndL, label=below:{\mbox{\footnotesize $(X_{af},Z_{af})$}}]  {a};
	\node[transitionL] (MOA)  [right of=down, node distance=1.25\ndL] {MaxOrdArithm};
	\node[placeL]      (plb) [right of=MOA, node distance=1.75\ndL,label=above:{\mbox{\footnotesize $(\ldots,\maxOrd(X), \MaxOrd(X))$}}]  {b};
	\node[transitionL] (descent)  [right of=pla, node distance=1.6\ndL] {Descent};
	\node[transitionL] (coll)  [right of=plb, node distance=2.15\ndL] {\parbox{32pt}{Glue \&\\ Collect}};
	\node[placeL]      (plc)   [below of=coll, node distance=1\ndL]  {c};
	\node[transitionL] (heur)  [left of=plc, node distance=1.45\ndL] {Heureka};
	\node[endplaceL]   (end)   [below of=heur, node distance=0.75\ndL,label=left:{\mbox{\footnotesize $(X_{af},Z_{af},\cB_{af}, \xch{\fP}_{af}, \maxNu(X_{af}),\MaxNu(X_{af}))$}}]    {};

	\path[->]
	(start) edge (down)
	(down) edge (pla)
	(pla)  edge (MOA)
	(MOA) edge (plb)
	(plb) edge node [right] {\footnotesize $\maxOrd(X) = 1 $} (descent)
	(descent)  edge node [above] {\footnotesize many} (pla)
	(plb) edge	node [above] {\footnotesize $\maxOrd(X) > 1 $} (coll)
	(plc)  edge [bend left=20]  node [left] {\footnotesize $X$ not covered} (coll)
	(coll) edge [bend left=20] (plc)
	(plc) edge  node [above] {\footnotesize $X$ covered} (heur)
	(heur) edge (end)
	;
\end{petrinet}
\caption{\xch{Petri net for {\tt MaxRefOrd.}}}
\label{Fig:MaxRefOrd}
\end{figure}

In this Petri net, we are still hiding another key step, the calculation of
the locus of maximal order in the arithmetic case, i.e., {\tt MaxOrdArith}.
This is the last of the Petri nets we are presenting here \xch{in Figure~\ref{Fig:MaxOrdArith}} and starts by
calling {\tt MaxOrd0}, already available functionality to compute the locus 
of maximal order in characteristic zero.

In the absence of interesting primes, this already terminates the algorithm.
Otherwise, the stalk at each interesting prime needs to be considered 
separately giving rise to at least $\#\xch{\fP}$ tokens at place b for each of 
which the prime $p$ has already been replaced by a symbol $P$. Note that at
each of the stalks a covering by minors of the Jacobian matrix may become
necessary. The computation of Hasse{-Schmidt} 
deriviatives in {\tt HasseDer} is a task for the back-end {\sc Singular}, as
are the computations after resubstitution of $P$ by $p$.

In the transition {\tt Glue\&Collect}, the computed maximal order of
the token from place d is compared to the maximal order assigned to the
token from place e (technically, we again need to make sure that place e holds
an initialization token corresponding to the empty set at the beginning,
but this is not shown here). According to the result of the comparison, 
the one of lower value is dropped or the union of the computed loci is 
formed in case of equality. This comprises a gluing step for charts by
maximal minors and the collecting of the stalks at primes.
At the same time, the prime $p$ is discarded from $\xch{\fP}$ allowing us
to leave the loop as soon as $\xch{\fP}$ is empty. The transition {\tt Drop} then
just drops unnecessary parts of the coloring to reach conformity with the
coloring of tokens at place a. 

\smallskip 

\begin{figure} 
\begin{petrinet}
	\node[startplaceL] (start) [label=right:{\mbox{\footnotesize $(X,Z,\cB, \xch{\fP}) = (X_{af},Z_{af},\cB_{af}, \xch{\fP}_{af})$}}]                  {};
	\node[transitionL] (maxnull)  [below of=start, node distance=0.7\ndL] {MaxOrd0};
	\node[placeL]      (pla) [below of=maxnull, node distance=0.7\ndL]  {a};
	\node[invisible]   (invi)  [below of=pla, node distance=0.3\ndL, label=right:{\mbox{\footnotesize $ \ (\ldots, maxord, MaxOrd)$}}] {};
	\node[transitionL] (heur)  [right of=pla, node distance=2\ndL] {Clean-Up};
	\node[endplaceL]   (end)   [below of=heur, node distance=0.7\ndL,label=below:{\mbox{\footnotesize $ (X_{af},Z_{af},\cB_{af}, \maxOrd(X_{af}),\MaxOrd(X_{af}) )$}}]   {};
	\node[transitionL] (split1)  [left of=pla, node distance=1.2\ndL] {\parbox{34pt}{Split \&\\ Replace}};
	\node[placeL]      (plb)   [below of=split1, node distance=\ndL,label=left:{\mbox{\footnotesize $(\ldots,p,J_X,J_Z,M)$}}]  {b};
	\node[transitionL] (hasse)  [below of=plb, node distance=0.85\ndL] {HasseDer};
	\node[placeL]      (plc)   [below of=hasse, node distance=0.85\ndL,label=left:{\mbox{\footnotesize $(\ldots,DiffList, m)$}}]   {c};
	\node[transitionL] (resub)  [below of=plc, node distance=0.85\ndL] {\parbox{56pt}{Resubstitute \\ \& Eliminate}};
	\node[placeL]      (pld)   [right of=resub, label=right:{\mbox{\footnotesize $(\ldots, maxord^{new}, MaxOrd^{new}, \ldots )$}}, node distance=1.2\ndL]  {d};
	\node[transitionL] (coll)  [above of=pld, node distance=0.85\ndL] {\parbox{34pt}{Glue \& \\ Collect}};
	\node[placeL]      (ple)   [above of=coll, node distance=0.85\ndL]  {e};
	\node[transitionL] (drop)  [above of=ple, node distance=0.9\ndL] {Drop};

	\path[->]
	(start) edge (maxnull)
	(maxnull) edge (pla)
	(pla) edge node [above] {\footnotesize $ \xch{\fP} \neq \emptyset $} (split1)
	(pla) edge node [above] {\footnotesize $ \xch{\fP} = \emptyset $} (heur)
	(heur) edge (end)
	(split1) edge node [left] {\footnotesize many}  (plb)
	(plb) edge (hasse)
	(hasse) edge (plc)
	(plc) edge (resub)
	(resub) edge (pld)
	(pld) edge (coll)
	(ple)  edge [bend left=20] node [right] {\footnotesize $\xch{\fP} \neq \emptyset \,|| \, X$ not covered} (coll)
	(coll) edge [bend left=20] (ple)
	(ple) edge node [right] {\footnotesize $\xch{\fP} = \emptyset \,\&\& \,X $ covered}  (drop)
	(drop) edge (pla)
	;
\end{petrinet}
\caption{\xch{Petri net for {\tt MaxOrdArith.}}}
\label{Fig:MaxOrdArith}
\end{figure}

\label{xch:auxiliary}

\section{Outlook}

\subsection{Normalization in the arithmetic case}

Up to this point, all considerations in the article have followed the 
general structure of Hironaka style desingularization, i.e., using suitably 
chosen smooth centers for resolving by blowing up. There is another classical
approach to resolution of singularities for surfaces (even in 
mixed characteristic): Lipman introduced a desingularization process in 
\cite{Lip}, which \xch{alternately applies the} normalization of the surface followed by \xch{blow ups of closed points.}

Over a field, normalization can be achieved algorithmically following the 
ideas of Grauert (see \cite{dJ}), for which state-of-the-art algorithms and 
implementations are available (see for instance \cite{GLS}, \cite{Betal}).
Crucial to all of \xch{this} is the use of the ideal of the singular locus as 
an approximation of the non-normal locus. 
In the arithmetic case, the same theoretical ideas still hold, but on the
practical side the approximation of the non-normal locus causes a problem:
if it is done by means of the vanishing of the minors of the Jacobian matrix, 
it is by far too coarse, as it does not only detect the singular points, but 
also fiber singularities.
Recall the first lines of Example 
\ref{ExFibreSing} for an illustration of the difference between singular point
and singularity of the fiber. 
\xch{The approximation of the 
non-normal locus hence needs to be done using the non-regular locus.}

\label{xch:Anne_to_do_non-normal}
\xch{Recalling the considerations in Observation \ref{Obs:d_x} and the subsequent
definition of $\nu_{\sf ref}$, we know that a scheme $X$ is non-regular at a point
$x$, if and only if $d_x$, the second entry of $\nu_{\sf ref}$, exceeds $1$. 
However, this is not as directly computable as it seems. We need to determine
the non-regular locus as
\[ 
	\Sigma := \bigcup_{i=1}^{N} {\mathcal V}_{(N-i,\geq 2)}(X),
\]
where ${\mathcal V}_{(N-i,\geq 2)}(X)$ describes the locus on $X$ where the invariant
has first value $N-i$ and second value at least $2$. For $i=N$, the locus 
${\mathcal V}_{(0,\geq 2)}(X)$ is precisely the locus of order at least $2$ on $X$, which is
computable as the vanishing locus of the second entry of the output list of
the Algorithm \ref{HasseDer} by Remark \ref{Rk:be_careful_diff_op}(3). Suitably
covering the complement of this locus by principal open sets and moving one 
generator of order $1$ of $I_X$ into $I_Z$ on each open set, as is discussed 
in more detail in the characteristic zero case in \cite{BDF}, we can then
compute the locus ${\mathcal V}_{(1,\geq 2)}(X)$ on each of the open sets and so on. This 
provides a parallel algorithm to determine the non-regular locus. }

\subsection{Schemes of higher dimension} 
As it is already mentioned in \cite{CJS} Remark \xch{6.29}, 
the formulation of the algorithm (see Construction \ref{CJS-Constr}) allows an 
{\em arbitrary} dimensional scheme $ X $ instead of a two-dimensional one.
But termination of the algorithm is not known in higher dimensions
(even in the situation over a field of characteristic zero),
neither for the original version nor for our variant.

While the newly created components in the locus of maximal singularity
(i.e., those with label 1 or larger) are well-behaved in dimension two,
there is less control in higher dimensions, 
see \cite{InvDim2} Example 2.5.
Furthermore, one has to be more careful when dealing with the labels of
irreducible components, see \cite{InvDim2} Example 2.6.
Since the mentioned examples are hypersurfaces they apply for both variants of the algorithm.

In the papers studying the algorithm
(\cite{HiroBowdoin}, \cite{CosToho}, \cite{CJS}, \cite{InvDim2}) 
Hironaka's characteristic polyhedron \cite{HiroCharPoly} 
plays a crucial role as a measure for the improvement of 
the singularity during the resolution process.
Another difficulty in higher dimension is to obtain a similar control on the characteristic polyhedron or to give an alternative invariant measuring the improvement.

Therefore an implementation of the algorithm could be extremely useful
for investigations in this direction.
In particular, it will be easier to test numerous examples on termination
of the algorithm, to search for patterns in order to develop new invariants,
and to explore the behavior of potential invariants.

\smallskip

Finally, one may think that choosing always the oldest irreducible
components seems quite arbitrary.
In fact, as the following example shows, we cannot modify the algorithm 
by considering always the newest irreducible component.

\begin{Ex}
	\label{Ex:loop}
	Let $ k $ be a field of characteristic zero and $ \cB= \emptyset $.
	Let $ X = V(f) \subset \mathbb{A}^5_k $ be the affine 
	variety given by the polynomial
	\[
	f = t^5 + x^4 y^2 z^4 + w^{10} z^6.
	\]
	We have $ \maxNu(X) =(5, 5 ) $ and 
	$ 
	\OMaxNu(X) =  V(t,x,z) \cup V(t,y,z) \cup V(t,x,y,w) .
	$ 
	All components get label $ 0 $ and the algorithm blows up
	with center the origin $ D_0 := V(t,x,y,z,w) $.
	
	In the affine chart with coordinates $ (t',x',y',z',w') = (\frac{t}{z}, \frac{x}{z}, \frac{y}{z}, z, \frac{w}{z}) $ the strict transform $ X' $ of $ X $ is given by 
	(for simplicity, we abuse notation and call the coordinates again $(t,x,y,z,w) $)
	\[
	f' = t^5 + x^4 y^2 z^5 + w^{10} z^{11}.
	\]
	We have $ \maxNu(X') = (5,5) $ and 
	$ 
	\OMaxNu(X') = V (t,z) \cup V(t,x,y,w).
	$
	Since $ V(t,z) $ is contained in the exceptional divisor
	and the center is not an entire irreducible component of $ \OMaxNu(X) $, the component $ V(t,z) $ gets the label $ 1 $.
	
	{\em Suppose we modify Construction \ref{CJS-Constr} such that
		we always consider the irreducible components with highest label.} 
	Then the center of the next blowing up is $ D_1 := V(t,z) $.
	
	In the affine chart with coordinates $ (t',x',y',z',w') = (\frac{t}{z}, x , y, z, w) $, we obtain that the strict transform $ X'' $ is given by
	\[
	f'' = t^5 + x^4 y^2 + w^{10} z^{6}.
	\]
	Further, $ \maxNu(X'') = (5,5) $, and 
	$ 
	\OMaxNu(X') = V (t,x,y,z) \cup V(t,x,y,w),
	$
	where $  V (t,x,y,z) $ has label $ 2 $ and $ V(t,x,y,w) $ has label $ 0 $.
	The modified algorithm chooses $ D_2 :=  V (t,x,y,z) $ as the next center.
	
	In the affine chart with coordinates  $ (t',x',y',z',w') = (\frac{t}{z}, \frac{x}{z}, \frac{y}{z}, z, w) $ the strict transform $ X^{(3)} $ is given by 
	\[
	f^{(3)} = t^5 + x^4 y^2 z + w^{10} z.
	\]
	Again, $ \maxNu(X^{(3)}) = (5,5) $, and we have
	$ 
	\OMaxNu(X^{(3)}) = V (t,x,z,w) \cup V(t,x,y,w),
	$
	where the component $  V (t,x,z,w) $ gets label $ 3 $ and $ V(t,x,y,w) $ has label $ 0 $.
	Hence $ D_3 :=  V (t,x,z,w) $ is the center of the next blowing up.
	
	In the affine chart with coordinates  $ (t',x',y',z',w') = (\frac{t}{z}, \frac{x}{z}, y, z, \frac{w}{z}) $ the strict transform $ X^{(4)} $ is given by 
	\[
	f^{(4)} = t^5 + x^4 y^2  + w^{10} z^6.
	\] 
	But $ f^{(4)} = f'' $. 
	Thus we have created a loop and the modified algorithm never ends.
\end{Ex}

\end{document}